\documentclass[11pt,letterpaper,reqno]{amsart}
\usepackage{fullpage}
\usepackage{amsmath,amsthm,amsfonts,amssymb,amscd}
\usepackage{empheq}
\usepackage{thmtools}
\usepackage{hyperref}
\usepackage{makecell}
\usepackage{enumerate}
\usepackage{enumitem}
\usepackage{bbm}
\usepackage{array}
\usepackage{float}
\usepackage{graphicx}

\usepackage{bbm}
\usepackage{stmaryrd}

\allowdisplaybreaks

\let\subsectiontemp\subsection
\renewcommand{\subsection}[1]{ 
    \subsectiontemp{#1} \hfill\vspace{0.5\linespacing} 
} 

\hypersetup{
  colorlinks=true,
  linkcolor=blue,
  citecolor=blue,
  linkbordercolor={0 0 1}
}

\newtheorem{theorem}{Theorem}[section]

\newtheorem{lemma}[theorem]{Lemma}

\newtheorem{proposition}[theorem]{Proposition}
\newtheorem{example}[theorem]{Example}

\mathtoolsset{showonlyrefs=true}
\numberwithin{equation}{section}
\numberwithin{figure}{section}
\numberwithin{table}{section}

\newcommand{\lrp}[1]{\left(#1\right)}
\newcommand{\lrb}[1]{\left[#1\right]}
\newcommand{\lrcb}[1]{\left\{#1\right\}}
\newcommand{\lrfloor}[1]{\left\lfloor #1 \right\rfloor}

\newcommand{\pttmatrix}[4]{
\left(
\begin{smallmatrix}
  #1 & #2 \\
  #3 & #4
\end{smallmatrix}
\right)
}

\newcommand{\ZZ}{\mathbb{Z}}
\newcommand{\CC}{\mathbb{C}}
\newcommand{\QQ}{\mathbb{Q}}
\newcommand{\RR}{\mathbb{R}}

\newcommand{\SL}{\text{SL}}

\newcommand{\FF}{\mathcal{F}}
\newcommand{\HH}{\mathcal{H}}
\newcommand{\CM}{\textnormal{CM}}
\newcommand{\RM}{\textnormal{RM}}

\author[E. Ross]{Erick Ross}
\address[E. Ross]{School of Mathematical and Statistical Sciences, Clemson University, Clemson, SC}
\email{erickjohnross@gmail.com}

\author[H. Xue]{Hui Xue}
\address[H. Xue]{School of Mathematical and Statistical Sciences, Clemson University, Clemson, SC}
\email{huixue@clemson.edu}



\title{Equidistribution of CM points and RM curves}

\begin{document}

\begin{abstract}
In 1988, William Duke showed that CM points of fundamental discriminant $D$ are equidistributed in the complex upper half-plane $\mathcal H$ as $D \to -\infty$. He also showed a similar result for RM curves (a positive discriminant analog of CM points).
In this paper, we investigate analogous problems concerning the distribution of CM points and RM curves along fixed geodesics in $\mathcal H$, and around fixed points in $\HH$. Specifically, we show that CM points and RM curves are equidistributed along every fixed rational geodesic in $\mathcal H$, and around every fixed CM point in $\HH$. To prove these results, we solve the aggregate Linnik problem for arbitrary binary quadratic forms.
\end{abstract}

\maketitle

\section{Introduction}

Let $\HH := \lrcb{z = x+iy \in \CC ~:~ y > 0}$ denote the complex upper half plane and $\SL_2(\ZZ)\backslash\HH$ denote the modular surface.
Then let $d\mu_{\text{hyp}} = \frac{3}{\pi} \frac{dx\,dy}{y^2}$ denote the hyperbolic measure on $\HH$ and $ds_\text{hyp} = \frac{\sqrt{dx^2+dy^2}}{y}$ denote the hyperbolic metric on $\HH$.

In \cite{duke}, Duke studied the distribution of CM points over $\HH$.
Specifically, let
 % temporary
$$\FF := \lrcb{z = x+iy \in \HH ~:~ |z| > 1, ~ -\frac{1}{2} \le x \le \frac 12 } \subseteq \HH$$ denote the standard fundamental domain for $\SL_2(\ZZ)\backslash\HH$. 
Then Duke showed that CM points of fundamental discriminant $D$ are $d\mu_\text{hyp}$-equidistributed over $\FF$ as $D \to -\infty$. 
Duke also showed a similar result for RM curves in $\FF$ (a positive discriminant analog of CM points). We note here that all the relevant objects are well-defined both in $\HH$, as well as in the quotient $\SL_2(\ZZ)\backslash\HH$. Hence Duke's work can equivalently be interpreted as an equidistribution result over all of $\HH$, or as an equidistribution result over $\SL_2(\ZZ)\backslash\HH$.

Now, observe that Duke's result concerns the distribution of CM points and RM curves in fixed $2$-dimensional regions in $\HH$.
In this paper, we investigate the analogous problem for $1$-dimensional and $0$-dimensional objects. Specifically, we study the distribution of CM points and RM curves along fixed geodesics in $\HH$, and around fixed points in $\HH$. Duke's results and our results are summarized in the following bulleted list for comparison.

\begin{itemize}
    \item Duke - \cite[Theorem 1 (i)]{duke}: CM points are $d\mu_\mathrm{hyp}$-equidistributed within fixed regions in $\HH$.
    \item Duke - \cite[Theorem 1 (ii)]{duke}: RM curves are $d\mu_\mathrm{hyp}$-equidistributed within fixed regions in $\HH$.
    \item Theorem \ref{thm:equid-CM-geodesic}: CM points are $ds_\mathrm{hyp}$-equidistributed along fixed rational geodesics in $\HH$.
    \begin{itemize}
        \item[$\circ$] This result only holds for rational geodesics since there exists at most one CM point along any given non-rational geodesic.
    \end{itemize}
    \item Theorem \ref{thm:equid-RM-geodesic}: RM curves are $ds_\mathrm{hyp}$-equidistributed along fixed rational geodesics in $\HH$.
    \begin{itemize}
        \item[$\circ$] This result only holds for rational geodesics since there exists at most one RM curve along any given non-rational geodesic.
    \end{itemize}
    \item Theorem \ref{thm:equid-CM-fixedpt}: CM points are $d\theta$-equidistributed around fixed points in $\HH$.
    \item Theorem \ref{thm:equid-RM-fixedpt}: RM curves are $d\theta$-equidistributed around fixed CM points in $\HH$.
    \begin{itemize}
        \item[$\circ$] This result only holds for CM points since there exists at most one RM curve around any given non-CM point. 
    \end{itemize}
\end{itemize}

As with Duke's work, our work can equivalently be interpreted as equidistribution results over $\HH$, or as equidistribution results over $\SL_2(\ZZ)\backslash\HH$; see Section \ref{sec:discussions} for further discussion.

Before stating our results precisely, we first give several definitions.

\textbf{(Binary quadratic forms)}  
A \textit{binary quadratic form} is any bivariate polynomial of the form $Q(x,y) = Ax^2+Bxy+Cy^2$ where $A,B,C \in \ZZ$. 
In this paper, we will mainly just be concerned with the corresponding univariate polynomial $Q(x,1) = Ax^2+Bx+C$ and its roots.
We will usually denote a binary quadratic form by its corresponding triple $(A,B,C)$. 
The \textit{discriminant} of $(A,B,C)$ is $D := B^2-4AC$. We say that a binary quadratic form is \textit{normalized} if $\gcd(A,B,C)=1$ and the first non-zero entry of the triple $(A,B,C)$ is positive. Note that any nonzero binary quadratic form can be normalized by scaling by an appropriate constant (and this normalization does not change, for example, the roots of $Ax^2 + Bx + C$).
% We note that 
% this notion of a normalized binary quadratic form is somewhat non-standard. However, these
% normalized binary quadratic forms
% are correct objects to study in our context because normalized binary quadratic forms exactly parametrize rational geodesics.

\textbf{(Geodesics)} The \textit{geodesics} in $\HH$ with respect to the hyperbolic metric are precisely the half-lines $G_x :=\{x+iy ~:~ y \in (0,\infty)\}$ for $x \in \RR$, and the semicircles $G_{q,r} := \{q + re^{i\theta} ~:~ \theta \in (0,\pi)\}$ for $q\in\RR,r\in \RR^+$. Throughout this entire paper, the word ``half-line" always refers to a vertical half-line in $\HH$ starting at the real axis, and the word ``semicircle" always refers to a semicircle in $\HH$ centered at a point on the real axis.
These geodesics in $\HH$ are precisely the preimages of geodesics in $\SL_2(\ZZ)\backslash\HH$.

\textbf{(Rational geodesics)} 
We say that a geodesic $G_x$ or $G_{q,r}$ is \textit{rational} if its real part $x$ is rational, or if its center $q$ and squared radius $r^2$ are rational, respectively. 
Note that rational geodesics are defined in this way because, for example, a geodesic contains more than one CM point if and only if it is rational.
It turns out that one can parametrize rational geodesics by normalized binary quadratic forms of positive discriminant. In particular, for any normalized binary quadratic form $(A,B,C)$ of positive discriminant, we define $G_{(A,B,C)}$ to be the geodesic with endpoints the root(s) of $Ax^2+Bx+C$. Explicitly, this means that
\begin{align}
    G_{(A,B,C)}
    &= \{z \in \HH ~:~ A|z|^2+B\,\mathrm{Re}(z) + C = 0 \} \label{eqn:GABC-set-formula}\\
    &= 
    \begin{cases}
        \{\lrp{\frac{-C}{B}}+iy ~:~ y \in (0,\infty)\} & \text{if } A=0, \\
        \{\lrp{\frac{-B}{2A}} + \lrp{\frac{\sqrt{B^2-4AC}}{2A}}e^{i\theta} ~:~ \theta \in (0,\pi)\} & \text{if } A>0.
    \end{cases} 
\end{align}
This map defines a bijection between normalized binary quadratic forms of positive discriminant and rational geodesics. We also define the discriminant of $G_{(A,B,C)}$ to be $D := B^2-4AC$.

\textbf{(CM points)} We define a \textit{CM point} to be the root $z \in \HH$ of $az^2+bz+c$, where $(a,b,c)$ is a normalized binary quadratic form of negative discriminant. We denote this CM point by $[a,b,c] := \frac{-b + \sqrt{b^2-4ac}}{2a}$. This parametrization gives a bijection between CM points and normalized binary quadratic forms of negative discriminant. Additionally, we define the discriminant of $[a,b,c]$ to be
the discriminant of the corresponding binary quadratic form.

\textbf{(RM curves)} RM curves are a positive discriminant analog of CM points.
% for binary quadratic forms of positive discriminant. 
Note that if $(a,b,c)$ has positive discriminant with $a \ne 0$, then the roots of $az^2+bz+c$ both lie on the real axis. So we define an \textit{RM curve} to be the semicircle in $\HH$ connecting the roots of $az^2+bz+c$, where $(a,b,c)$ is a normalized binary quadratic form of positive discriminant with $a \ne 0$. We denote this RM curve by 
$$\langle a,b,c\rangle := \lrcb{\text{the semicircle connecting }  \frac{-b \pm \sqrt{b^2-4ac}}{2a}},$$ 
and define the discriminant of $\langle a,b,c \rangle$ to be 
the discriminant of the corresponding binary quadratic form.
We also note here that RM curves happen to be precisely the set of all semicircle rational geodesics in $\HH$ (i.e. the rational geodesics $G_{(a,b,c)}$ with $a\ne0$).

\textbf{(CM points and RM curves along a geodesic)} In this paper, we are interested in studying the distribution of CM points and RM curves \textit{along} a fixed rational geodesic $G$. For CM points, it is obvious what this means; we are interested in the CM points that happen to lie on $G$. However, it is slightly less obvious how one should define RM curves along $G$. It turns out that the natural objects to study are the RM curves that have perpendicular-intersection with $G$. So to be precise, when we discuss the distribution of RM curves \textit{along} $G$, we are referring to the distribution of RM curve perpendicular-intersection points on $G$.

\textbf{(Equidistribution)} 
Let $I$ be an interval and $\mu$ be a Radon measure on $I$. Then a sequence $\{\alpha_n\}_{n \ge 1} \subseteq I$ is said to be \textit{equidistributed over $I$} with respect to the measure $\mu$ if 
\begin{align}
    \frac{ 
        \#\, I_1\cap\{\alpha_n\}_{n \le N} 
    }{
        \#\, I_2\cap\{\alpha_n\}_{n \le N}
    } 
    \longrightarrow \frac{\mu(I_1)}{\mu(I_2)}  \quad \text{as}\quad N \to \infty
\end{align}
for all compact subintervals $I_1 \subseteq I_2 \subseteq I$ with $\mu(I_2) \ne 0$.

% Let $M$ be a smooth manifold and $\mu$ be a Radon measure on $M$. Then a sequence $\{\alpha_n\}_{n \ge 1} \subseteq M$ is said to be \textit{$\mu$-equidistributed} over $M$ if 
% \begin{align}
%     \frac{ 
%         \#\, \Omega_1\cap\{\alpha_n\}_{n \le N} 
%     }{
%         \#\, \Omega_2\cap\{\alpha_n\}_{n \le N}
%     } 
%     \longrightarrow \frac{\mu(\Omega_1)}{\mu(\Omega_2)}  \quad \text{as}\quad N \to \infty
% \end{align}
% for all compact $\mu$-continuity sets $\Omega_1 \subseteq \Omega_2 \subseteq M$ with $\mu(\Omega_2) \ne 0$. 

\textbf{(Hyperbolic equidistribution over a geodesic)} 
In this paper, we show $ds_\text{hyp}$-equidistribution of certain sequences over a given geodesic. We note here how the hyperbolic metric restricts to each type of geodesic.
Points on half-line geodesics $G_x$ are parametrized by their imaginary part, $y = \mathrm{Im} (x+iy)$. In this case, the hyperbolic metric on $G_x$ restricts to $ds_\text{hyp} = \frac{1}{y}dy$.
Points on semicircle geodesics $G_{q,r}$, on the other hand, are parametrized by their arguments (around $q$), $\theta = \mathrm{arg}_q(q+re^{i\theta})$. In this case, the hyperbolic metric on $G_{q,r}$ restricts to $ds_\text{hyp} = \frac{1}{\sin \theta}d\theta$. Hyperbolic equidistribution over a geodesic is then defined via this parametrization, using the above definition.

\textbf{(Angle equidistribution around a point)} 
For a fixed point $p \in \HH$, we also study the distribution of certain points and geodesics around $p$. The geodesics $G$ passing through $p$ can be parametrized by  $\mathrm{ang}_{p}(G)$: the angle at which they pass through $p$, where $\theta=0$ represents the angle pointing straight down towards the real axis.
Note that $\mathrm{ang}_{p}(G)$ can alternatively be interpreted as the argument of $p$ as a point on $G$:
\begin{align}
    \mathrm{ang}_{p}(G_{q,r}) = \arg_q(p) \quad \text{and} \quad \mathrm{ang}_{p}(G_{x}) = 0.
\end{align}

Similarly, for points $z$ around $p$, their corresponding angle from $p$ can be computed as 
\begin{align}
    \mathrm{ang}_{p}(z) &= 
    \begin{cases}
        \mathrm{ang}_{p}(G) & \text{if}\quad  \mathrm{Re}(z) > \mathrm{Re}(p) \text{ or } \mathrm{Re}(z) = \mathrm{Re}(p),\mathrm{Im}(z)<\mathrm{Im}(p) \\
        \mathrm{ang}_{p}(G)+\pi & \text{otherwise},
    \end{cases} \\
    &\text{where $G$ is the unique geodesic connecting $p$ and $z$}.
\end{align}
See Figure \ref{fig:geodesic-angle} for an illustration of these notions of $\mathrm{ang}_p(G)$ and $\mathrm{ang}_p(z)$.
We will show in Theorem \ref{thm:equid-RM-fixedpt} that the RM curves passing though CM points $p$ are equidistributed with respect to the uniform angle measure $d\theta$ (recall that RM curves are just special types of geodesics).
And we will also show in Theorem \ref{thm:equid-CM-fixedpt} that CM points around $p$ are equidistributed with respect to the uniform angle measure $d\theta$.

\par \raggedbottom
\begin{figure}[H]
    \centering
    \includegraphics[width=0.9\textwidth]{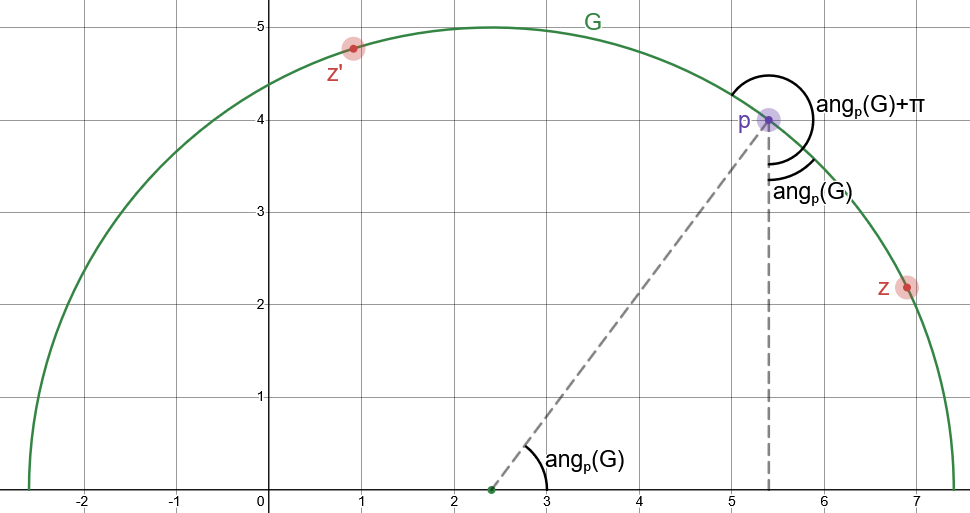}
    \caption{The angles $\mathrm{ang}_{p}(z) = \mathrm{ang}_{p}(G)$ and $\mathrm{ang}_{p}(z') = \mathrm{ang}_{p}(G) + \pi$.}
    \label{fig:geodesic-angle}
\end{figure}
\par \flushbottom

We are now finally able to state our results precisely.
\begin{theorem} \label{thm:equid-CM-geodesic}
    Fix a rational geodesic $G$, and let 
    $\CM^{G}_{\Delta}$ denote the set of CM points along $G$ with discriminant $|D| \le \Delta$. 
    Then as $\Delta \to \infty$, $\CM^{G}_{\Delta}$ is equidistributed along $G$ with respect to the hyperbolic metric.
\end{theorem}

We make three remarks about Theorem \ref{thm:equid-CM-geodesic}.
First, we remark that we had previously shown in \cite[Theorem 1.1]{boundary-CM-points} that CM points are $ds_\text{hyp}$-equidistributed on the boundary $\partial \FF$ of the fundamental domain for $\SL_2(\ZZ)\backslash\HH$. That result can be viewed as a special case of Theorem \ref{thm:equid-CM-geodesic}.
Second, observe that this equidistribution theorem differs from Duke's result \cite[Theorem 1]{duke} in that we are considering $\CM$ points for all discriminants $|D| \le \Delta$, instead of for individual discriminants $D$. This is necessary because the sets of CM points for individual discriminants $D$ do not even become dense along a given geodesic. 
Third, we remark that our result shows equidistribution of CM points of general discriminant, whereas \cite[Theorem 1]{duke} shows equidistribution just for CM points of fundamental discriminant.

Next, we show the equidistribution of RM curves along a fixed rational geodesic.
\begin{theorem} \label{thm:equid-RM-geodesic}
    Fix a rational geodesic $G$, and let 
    $\RM^{G}_{\Delta}$ denote the set of RM curves along $G$ with discriminant $D \le \Delta$. 
    Then as $\Delta \to \infty$, $\RM^{G}_{\Delta}$ is equidistributed along $G$ with respect to the hyperbolic metric.
\end{theorem}
Note that this theorem claims the equidistribution of RM curves along $G$.
To be precise, recall that this terminology is referring to the equidistribution along $G$ of the intersection points for RM curves that have perpendicular intersection with $G$.

Next, we show equidistribution of RM curves around a fixed CM point. 
\begin{theorem} 
    \label{thm:equid-RM-fixedpt}
    Fix a CM point $p$, and let 
    $\RM^p_{\Delta}$ denote the set of RM curves passing through $p$ with discriminant $D \le \Delta$. 
    Then as $\Delta \to \infty$, the angles $\mathrm{ang}_p\, \RM^p_{\Delta}$ are equidistributed with respect to the uniform angle measure.
\end{theorem}

Finally, in Theorem \ref{thm:equid-CM-fixedpt}, we prove equidistribution of CM points around any fixed point in $\HH$. It turns out that this theorem follows from Duke's original equidistribution result. For this reason, we have chosen to state Theorem \ref{thm:equid-CM-fixedpt} in terms of CM points of fundamental discriminant $D$. However, we remark here that in symmetry with Theorems \ref{thm:equid-CM-geodesic} - \ref{thm:equid-RM-fixedpt}, the same result also holds for CM points of discriminant $|D| \le \Delta$. This follows from an aggregate version of Duke's result; specifically the aggregate Linnik problem for $Q(a,b,c)=4ac-b^2$ (see the discussion above Theorem \ref{thm:equid-WABC}).
\begin{theorem} 
    \label{thm:equid-CM-fixedpt}
    Fix $z_0 \in \HH$ and $s_0 \in \RR^+$, and for negative fundamental discriminants $D$, let 
    $\CM^{\overline B(z_0, s_0)}_{D}$ denote the set of CM points of discriminant $D$
    within hyperbolic distance $s_0$ of $z_0$.
    Then as $D \to -\infty$, the angles $\mathrm{ang}_{z_0}\, \CM^{\overline B(z_0,s_0)}_{D}$ are equidistributed with respect to the uniform angle measure.
\end{theorem}

Near the beginning  of the introduction, we gave a bulleted list showing the symmetry between \cite[Theorem 1]{duke} and Theorems \ref{thm:equid-CM-geodesic} - \ref{thm:equid-CM-fixedpt}. We would also like to point out four additional symmetries between Theorems \ref{thm:equid-CM-geodesic} - \ref{thm:equid-RM-fixedpt}.
\begin{enumerate}
    \item 
    Theorems \ref{thm:equid-CM-geodesic} and \ref{thm:equid-RM-geodesic} show that CM points/RM curves are equidistributed along each fixed rational geodesic. From the (in some sense) inverse perspective, Theorems \ref{thm:equid-RM-geodesic} and \ref{thm:equid-RM-fixedpt} show that rational geodesics are equidistributed around each fixed CM point/RM curve. Note that as stated, Theorems \ref{thm:equid-RM-geodesic} and \ref{thm:equid-RM-fixedpt} just show the equidistribution of RM curves. However, these results can also be viewed as showing equidistribution of rational geodesics since there is at most one half-line geodesic around each fixed CM point/RM curve.
    \item 
    Theorems \ref{thm:equid-CM-geodesic} - \ref{thm:equid-RM-fixedpt} show that CM points are equidistributed along every fixed RM curve, RM curves are equidistributed along every fixed rational geodesic, and rational geodesics are equidistributed around every fixed CM point.
    \item 
    The CM points/RM curves in Theorems \ref{thm:equid-CM-geodesic} - \ref{thm:equid-RM-fixedpt} are all governed by the same Diophantine equation (see Lemma \ref{lem:CM-RM-along-G-diophantine}).
    \item 
    Theorems \ref{thm:equid-CM-geodesic} - \ref{thm:equid-RM-fixedpt} all follow from some form of Theorem \ref{thm:equid-WABC}, discussed below.
\end{enumerate}

To show the above results, we solve the aggregate Linnik problem for arbitrary binary quadratic forms. We give a brief overview of what exactly this terminology means.

Fix a homogeneous integer polynomial $Q(x_1, \ldots x_n)$, and let $$ [(x_1,\ldots x_n)]_Q := \frac{(x_1, \ldots x_n)}{Q(x_1,\ldots x_n)^{1/\mathrm{deg}\, Q}}$$ denote the normalization of the point $(x_1,\ldots x_n) \in \RR^n$ onto the surface $$S_Q := \lrcb{(x_1,\ldots x_n) \in \RR^n : Q(x_1,\ldots x_n)=1}.$$  Then the \textit{classical Linnik problem for $Q$} is to determine the distribution of the points
\begin{align}
    \Big\{[(x_1, \ldots x_n)]_Q : (x_1,\ldots x_n) \in \mathbb{Z}^n,\    Q(x_1,\ldots x_n) = d \Big\}
\end{align}
over $S_Q$ as $d \to \infty$. The Linnik problems have attracted much attention over the years, including recent work by Einsiedler-Lindenstrauss-Michel-Venkatesh \cite{LP-einsiedler-lindenstrauss-michel-venkatesh-2011,LP-einsiedler-lindenstrauss-michel-venkatesh-2012}, Bourgain-Rudnick-Sarnak \cite{LP-bourgain-sarnak-rudnick-1,LP-bourgain-rudnick-sarnak-2}, Aka-Einsiedler-Shapira \cite{LP-aka-einsiedler-shapira},  Khayutin \cite{LP-khayutin}, Blomer-Brumley-Radziwi{\l}{\l} \cite{LP-blomer-brumley-radziwill}, and many others. We also note here that Duke's results on the equidistribution of CM points and RM curves have historically been some of the most important examples of Linnik problems; \cite[Theorem 1 (i) and (ii)]{duke} can be interpreted as the Linnik problems for $Q(a,b,c)=4ac-b^2$ and $Q(a,b,c)=b^2-4ac$, respectively. 

Analogously to the classical Linnik problem for $Q$, the \textit{aggregate Linnik problem for $Q$} is the problem of determining the distribution of the points
\begin{align}
    \Big\{[(x_1, \ldots x_n)]_Q : (x_1,\ldots x_n) \in \mathbb{Z}^n,\   0 < Q(x_1,\ldots x_n) \le \Delta \Big\}
\end{align}
over $S_Q$ as $\Delta \to \infty$. Here, we solve the aggregate Linnik problem for arbitrary binary quadratic forms (in fact, even allowing for non-integer binary quadratic forms). We plan to continue studying the aggregate Linnik problems in general in a future paper \cite{future-paper}. 
\begin{theorem} \label{thm:equid-WABC}
Fix $(A,B,C) \in \RR^3$ with $(A,B) \ne (0,0)$, and let $I \subseteq \widehat \RR$ be the subinterval of the projectively extended reals given by $I = \{t \in \widehat \RR ~:~ At^2+Bt+C > 0\}$. Then as $\Delta \to \infty$,
\begin{align}
    W_{\Delta}^{(A,B,C)} := \Big\{\frac{m}{n} ~:~ n\ge 1,~ \gcd(m,n)=1,~   0 < A m^2 + B mn + C n^2  \le \Delta
    \Big\}
\end{align}
is equidistributed over $I$ with respect to the measure $d\mu = \frac{1}{At^2+Bt+C} \,dt$.
\end{theorem}

Observe that for convenience, we have chosen to write the elements of $W_{\Delta}^{(A,B,C)}$ here as points $\frac{m}{n}$ in the projectively extended reals $\widehat \RR$ instead of as points on $S_{(A,B,C)}$. 
This identification comes from normalizing the lattice points $(x,y)$ onto the surface $\lrcb{(x,y) : y=1}$ instead of onto the surface $S_{(A,B,C)}$. Of course, this identification does not include the point $[(1,0)]_{(A,B,C)}$ (corresponding to the point $\pm \infty \in \widehat \RR$), but this does not affect the distribution in question since $[(1,0)]_{(A,B,C)}$ is only a single point.

The proof of Theorem \ref{thm:equid-WABC} is divided up into six different cases (depending on the signs of $A$, $B$, and $D := B^2-4AC$). Although the final result Theorem \ref{thm:equid-WABC} can be stated in a uniform way, these different cases require separate proofs because each scenario yields a different formula for $I$ and for the integral of $d\mu$. 

\begin{itemize}[leftmargin=*]
\item Theorem \ref{thm:equid-W--A=0,B>0}  shows $d\mu$-equidistribution over $I = (\frac{-C}{B},\infty)$. 
\item Theorem \ref{thm:equid-W--A=0,B<0} shows $d\mu$ equidistribution over $I = (-\infty, \frac{-C}{B})$.  
\item Theorem \ref{thm:equid-W--A>0,D>0} shows $d\mu$-equidistribution over $I = (\frac{-B+\sqrt{D}}{2A}, \infty] \cup [-\infty, \frac{-B-\sqrt{D}}{2A})$. 
\item Theorem \ref{thm:equid-W--A>0,D<0} shows $d\mu$-equidistribution over $I = [-\infty, \infty]$. 
\item Theorem \ref{thm:equid-W--A>0,D=0} shows $d\mu$-equidistribution over $I = (\frac{-B}{2A}, \infty] \cup [-\infty, \frac{-B}{2A})$. 
\item Theorem \ref{thm:equid-W--A<0,D>0} shows $d\mu$-equidistribution over $I = (\frac{-B+\sqrt{D}}{2A}, \frac{-B-\sqrt{D}}{2A})$.
\end{itemize}
Note that as subintervals of the projectively extended reals, the intervals $I=(\frac{-B+\sqrt{D}}{2A},\infty] \cup [-\infty,\frac{-B-\sqrt{D}}{2A})$ and $I = (\frac{-B}{2A}, \infty] \cup [-\infty, \frac{-B}{2A})$ here are understood to wrap around from $\infty$ to $-\infty$.

%%%%%%%%%%%%%%%%%%%%%%%%%%%%%%%%%%%%%

Finally, we give an overview of the paper. 
In Section \ref{sec:dioph-eqn-governing-CM-RM}, we prove Lemma \ref{lem:CM-RM-along-G-diophantine}, which gives a certain Diophantine equation governing the CM points/RM curves in Theorems \ref{thm:equid-CM-geodesic} - \ref{thm:equid-RM-fixedpt}. Next in Section \ref{sec:proof-main-theorems}, we prove Theorems \ref{thm:equid-CM-geodesic} - \ref{thm:equid-RM-fixedpt}, assuming Theorem \ref{thm:equid-WABC}. These proofs of these three theorems are all fairly similar. Next, in Section \ref{sec:proof-of-thm-equid-CM-fixedpt}, we carry out several geometric calculations in $\HH$ and prove Theorem \ref{thm:equid-CM-fixedpt}.
Next, in Sections \ref{sec:equid-WABC--ANT-lemmas} - \ref{sec:equid-WABC--A<0}, we carry out all of the technical equidistribution calculations to prove Theorem \ref{thm:equid-WABC}. These four sections are more or less independent from the rest of the paper (in particular, they will make no reference to  geodesics, CM points, RM curves, or any other objects on $\HH$). In Section \ref{sec:equid-WABC--ANT-lemmas}, we give several analytic number theory lemmas needed for the proof of Theorem \ref{thm:equid-WABC}. 
Then in Sections \ref{sec:equid-WABC--A=0}, \ref{sec:equid-WABC--A>0}, and \ref{sec:equid-WABC--A<0}, we prove Theorem \ref{thm:equid-WABC} for $A = 0$, for $A > 0$, and for $A < 0$, respectively. 
Lastly, in Section \ref{sec:discussions}, we discuss applications of our results.

\section{A Diophantine equation governing the relevant CM points/RM curves} \label{sec:dioph-eqn-governing-CM-RM}

In this section, we prove Lemma \ref{lem:CM-RM-along-G-diophantine}. This lemma shows that the CM points/RM curves in Theorems \ref{thm:equid-CM-geodesic} - \ref{thm:equid-RM-fixedpt} are all governed by the Diophantine equation $2aC+2cA=bB$.

\begin{lemma} \label{lem:CM-RM-along-G-diophantine} ~
\begin{enumerate}
\item 
Fix a rational geodesic $G_{(A,B,C)}$. Then $[a,b,c]$ lies on $G_{(A,B,C)}$ iff  
\begin{align} \label{eqn:cond--CM-on-geo}
2aC+2cA=bB.
\end{align}
\item Fix a rational geodesic $G_{(A,B,C)}$. Then $\langle a,b,c \rangle$ has  perpendicular intersection with $G_{(A,B,C)}$ iff
\begin{align} \label{eq:cond--RM-on-geo}
2aC+2cA=bB.
\end{align}
\item 
Fix a CM point $[A,B,C]$. Then $\langle a,b,c \rangle$ passes through $[A,B,C]$ iff 
\begin{align} \label{eqn:cond--RM-on-CM}
2aC+2cA=bB.
\end{align}
\end{enumerate}
\end{lemma}
\begin{proof} For part (1), recall that $G_{(A,B,C)} = \lrcb{z \in \HH ~:~ A|z|^2 + B\,\mathrm{Re}(z) + C = 0}$. For $z = [a,b,c]$, observe that
\begin{align}
    A|z|^2 + B\,\mathrm{Re}(z) + C 
    &= A\left| \frac{-b+\sqrt{b^2-4ac}}{2a}\right|^2 + B\,\mathrm{Re}\!\lrp{\frac{-b+\sqrt{b^2-4ac}}{2a}} +C \\
    &= A\lrp{ \frac{b^2}{4a^2} + \frac{4ac-b^2}{4a^2}} + B\lrp{\frac{-b}{2a}} +C \\
    &= A \frac{c}{a} - B \frac{b}{2a} + C \\
    &= \frac{2aC + 2cA - bB}{2a}.
\end{align}
Hence $z = [a,b,c]$ lies on $G_{(A,B,C)}$ iff $2aC + 2cA = bB$.

For part (2) for half-line rational geodesics, recall that $G_{(0,B,C)}$ is the half-line with real part $\frac{-C}{B}$. Then it is easy to see that $\langle a,b,c \rangle$ has perpendicular intersection with $G_{(0,B,C)}$ iff $\langle a,b,c \rangle$ is centered at $\frac{-C}{B}$. This occurs precisely when $\frac{-b}{2a} = \frac{-C}{B}$, i.e. when $2aC + 2c(0) = bB$.

For part (2) for semicircle rational geodesics, let $L$ denote the line segment connecting the centers $G_{(A,B,C)}$ and $\langle a,b,c \rangle$. Then note that $\langle a,b,c \rangle$ has perpendicular intersection with $G_{(A,B,C)}$ precisely when $L$ is the hypotenuse of a right triangle with legs the radii $G_{(A,B,C)}$ and $\langle a,b,c \rangle$. Now, we have that
\begin{align}
    &\text{radius}(G_{(A,B,C)})^2 + \text{radius}(\langle a,b,c \rangle)^2 -  |L|^2  \\
    &= \frac{B^2-4AC}{4A^2} + \frac{b^2-4ac}{4a^2} - \lrp{\frac{-b}{2a} - \frac{-B}{2A}}^2 \\
    &= \frac{4a^2 (B^2-4AC)}{16a^2A^2} + \frac{4A^2 (b^2-4ac)}{16a^2A^2} - \lrp{\frac{2aB-2bA}{4aA}}^2 \\
    &= \frac{4a^2 (B^2-4AC) + 4A^2 (b^2-4ac) - 4a^2B^2 - 4b^2A^2 + 8abAB}{16a^2A^2} \\
    &= \frac{8abAB - 16a^2AC - 16acA^2}{16a^2A^2} \\
    &= \frac{bB - 2aC - 2cA}{2aA}.
\end{align}
Hence by the Pythagorean theorem, $\langle a,b,c \rangle$ has perpendicular intersection with $G_{(A,B,C)}$ precisely when $2aC+2cA=bB$.

Part (3) follows immediately from part (1).
% For part (3), we already have that $\langle a,b,c \rangle$ passes through $[A,B,C]$ iff $2Ac + 2Ca = Bb$ (from part (1)). This condition is then equivalent to the equality $2aC + 2cA = bB$.
\end{proof}

\section{Proofs of Theorems \ref{thm:equid-CM-geodesic} - \ref{thm:equid-RM-fixedpt}.}
\label{sec:proof-main-theorems}

In this section, we prove Theorems \ref{thm:equid-CM-geodesic} - \ref{thm:equid-RM-fixedpt}, assuming Theorem \ref{thm:equid-WABC}.

{
\renewcommand{\thetheorem}{\ref{thm:equid-CM-geodesic}}
\begin{theorem} 
    Fix a rational geodesic $G$, and let 
    $\CM^{G}_{\Delta}$ denote the set of CM points along $G$ with discriminant $|D| \le \Delta$. 
    Then as $\Delta \to \infty$, $\CM^{G}_{\Delta}$ is equidistributed along $G$ with respect to the hyperbolic metric.
\end{theorem}
\addtocounter{theorem}{-1}
}
\begin{proof} ~We divide the proof into two cases, according to whether $G$ is a half-line or a semicircle. 

\noindent\textbf{Case 1: $G$ is a half-line rational geodesic.}

Writing $G$ as $G_{(0,B_0,C_0)}$, recall that we are trying to show that $\mathrm{Im}\, \CM^{G}_{\Delta}$ is equidistributed over $(0,\infty)$ with respect to the metric $ds_\text{hyp} = \frac{1}{y}dy$.

Now, Lemma \ref{lem:CM-RM-along-G-diophantine} states that a CM point $[a,b,c]$ lies on $G_{(0,B_0,C_0)}$ iff $2aC_0=bB_0$. So define $P,Q$ as
\begin{align}
    (P,Q) := \tfrac{\mathrm{sgn}(B_0)}{\gcd(B_0,2)} (2C_0,B_0),
\end{align}
and observe that $\gcd(P,Q)=1$ and $Q \ge 1$. 
Also, $2aC_0=bB_0$ iff $aP=bQ$, so we have that
\begin{align}
    \CM^{G}_{\Delta}  
    &= \Big\{ [a,b,c] ~:~ a \ge 1,~ \gcd(a,b,c)=1,~  aP=bQ,~
    0 > b^2-4ac \ge -\Delta \Big\}.
\end{align}

Then parameterizing the solutions to the Diophantine equation $aP = bQ$, we have that
\begin{align} 
    &\Big\{ (a,b,c) ~:~ a \ge 1,~ \gcd(a,b,c)=1,~ 
    aP=bQ
    \Big\} \\
    &= \lrcb{ (nQ,nP,-m) ~:~ n \ge 1,~ \gcd(m,n)=1 }.
\end{align}
Hence
\begin{align}
    \CM^{G}_{\Delta} 
    &= \Big\{  [nQ,nP,-m]  ~:~ n \ge 1,~ \gcd(m,n)=1,~  0 > n^2P^2 + 4Qnm \ge -\Delta
    \Big\} \\
    &= \Big\{  [nQ,nP,-m]  ~:~ n \ge 1,~ \gcd(m,n)=1,~  0 > B mn + C n^2 \ge -\Delta
    \Big\}, \label{eqn:CM-G0BC-set-formula} \\
    \text{where}&\quad  B := 4Q, \qquad C := P^2,
\end{align}
and so
\begin{align}
    \mathrm{Im}\,\CM^{G}_{\Delta} 
    &= \Big\{  \frac{\sqrt{-4Qnm -P^2n^2}}{2Qn}  ~:~ n \ge 1,~ \gcd(m,n)=1,~  0 > B mn + C n^2 \ge -\Delta
    \Big\} \\
    &= \Big\{  \sqrt{\frac{-4}{B} \lrp{\frac mn} - \frac{4C}{B^2}}  ~:~ n \ge 1,~ \gcd(m,n)=1,~  0 > B mn + C n^2 \ge -\Delta
    \Big\} \\
    &=  \Big\{  y\lrp{\frac mn}  ~:~ n \ge 1,~ \gcd(m,n)=1,~  0 > B mn + C n^2 \ge -\Delta
    \Big\} \\
    &= y(W_{\Delta}^{(0,-B,-C)}), \\
    \text{where }\ y(t) 
    &:= \sqrt{\frac{-4}{B} t - \frac{4C}{B^2}}.
\end{align}

Finally, we know from Theorem \ref{thm:equid-W--A=0,B<0} that $W_{\Delta}^{(0,-B,-C)}$ is $d\mu = \frac{1}{-Bt-C}$ equidistributed over $(-\infty,\frac{-C}{B})$. And under the reparametrization 
$y(t) = \sqrt{\frac{-4}{B} t - \frac{4C}{B^2}}$, this measure $d\mu = \frac{1}{-Bt-C}$ over $t \in (-\infty,\frac{-C}{B})$ becomes $d\mu = \frac{2}{B} \frac{1}{y} dy$ over $y \in (0,\infty)$. Thus we have that $\mathrm{Im}\,\CM^{G}_{\Delta} = y(W_{\Delta}^{(0,-B,-C)})$ is  equidistributed over $(0,\infty)$ with respect to the hyperbolic metric $\frac 1y dy$, completing the proof.

\noindent\textbf{Case 2: $G$ is a semicircle rational geodesic.}

Writing $G$ as $G_{(A_0,B_0,C_0)}$, let $D_0 = B_0^2-4A_0 C_0$, $q=\frac{-B_0}{2A_0}$, $r = \frac{\sqrt{D_0}}{2A_0}$, and recall we are trying to show that $\mathrm{arg}_q\, \CM^{G}_{\Delta}$ is equidistributed over $(0,\pi)$ with respect to the hyperbolic metric $ds_\text{hyp} = \frac{1}{\sin(\theta)} d\theta$.

Now, Lemma \ref{lem:CM-RM-along-G-diophantine} states that a CM point $[a,b,c]$ lies on $G_{(A_0,B_0,C_0)}$ iff $2aC_0+2cA_0=bB_0$. So define $P,Q,R$ as
\begin{align} \label{eqn:def-PQR}
    (P,Q,R) := \tfrac{1}{\gcd(D_0,2)} (-2C_0,-B_0,2A_0) 
\end{align}
and observe that $\gcd(P,Q,R)=1$ and $R \ge 1$.
Also, $2aC_0+2cA_0=bB_0$ iff $aP=bQ+cR$, so we have that
\begin{align}
    \CM^{G}_{\Delta}  
    &= \Big\{ [a,b,c] ~:~ a \ge 1,~ \gcd(a,b,c)=1,~  aP=bQ+cR,~
    0 > b^2-4ac \ge -\Delta \Big\}.
\end{align}

Now, let $S = \gcd(Q,R)$ and choose $b_0,c_0$ such that $b_0Q+c_0R = S$.
Parameterizing the solutions to the Diophantine equation $aP = bQ+cR$, we have that
\begin{align} 
    &\Big\{ (a,b,c) ~:~ a \ge 1,~ \gcd(a,b,c)=1,~ 
    aP=bQ+cR
    \Big\} \\
    &= \lrcb{ (nS,~nPb_0+m\tfrac{R}{S},~nPc_0-m\tfrac{Q}{S}) ~:~ n \ge 1,~ \gcd(m,n)=1 }.
\end{align}
%(This equality can also be verified via a two-way inclusion argument.)
Hence
\begin{align}
    \CM^{G}_{\Delta} 
    &= \Big\{  [nS,~nPb_0+m\tfrac{R}{S},~nPc_0-m\tfrac{Q}{S}]  ~:~ n \ge 1,~ \gcd(m,n)=1,\\
    &\hspace{23mm}  0 >(nPb_0+m\tfrac{R}{S})^2-4(nS)(nPc_0-m\tfrac{Q}{S}) \ge -\Delta
    \Big\} \\
    &= \Big\{[nS,~nPb_0+m\tfrac{R}{S},~nPc_0-m\tfrac{Q}{S}] ~:~ n\ge 1,~ \gcd(m,n)=1,~ \\
    &\hspace{50mm} 0 > A m^2 + B mn + C n^2  \ge -\Delta
    \Big\} \label{eqn:CM-GABC-set-formula} \\
    \text{where}\quad  A&:=\tfrac{R^2}{S^2}, \qquad B := 2Pb_0 \tfrac{R}{S}+4Q, \qquad C := P^2b_0^2-4SPc_0,  \label{eqn:def-ABC-for-CM^G-set-def}\\
     D &:= B^2 -4A C = 16Q^2+16PR = \frac{16}{\gcd(D_0,2)^2} D_0 > 0.
\end{align}

Now, for an angle $\theta=\arg_q [nS,~nPb_0+m\tfrac{R}{S},~nPc_0-m\tfrac{Q}{S}]$ in the set $\mathrm{arg}_q\, \CM^{G}_{\Delta}$, observe that 
\begin{align}
    q+r \cos\theta = \frac{-(nPb_0+m \tfrac{R}{S})}{2(nS)} = - \frac{Pb_0}{2S} - \frac{R}{2S^2} \frac{m}{n},
\end{align}
and so
\begin{align}
    \arg_q \CM^{G}_{\Delta} 
    &= \Big\{\theta\!\lrp{\frac{m}{n}} ~:~ n\ge 1,~ \gcd(m,n)=1,~   0 >A m^2 + B mn + C n^2  \ge -\Delta
    \Big\} \\
    &= \theta(W_{\Delta}^{(-A,-B,-C)}), \\
    \text{where }\ \theta(t) &:= \arccos\lrp{ \frac 1 r \lrb{- \frac{Pb_0}{2S} - \frac{R}{2S^2} t - q}}.
\end{align}

Also note that
\begin{align} 
    q = \frac{Q}{R}, &\qquad r^2-q^2 = \frac{P}{R}, \qquad r=\frac{\sqrt{Q^2+PR}}{R}, \label{eqn:q-r-formulas-via-PQR} \\
    \text{so }\ 
    \theta(t) 
    &= \arccos\lrp{ \frac{R}{\sqrt{Q^2+PR}} \lrb{ 
        \frac{-Pb_0}{2S} - \frac{R}{2S^2} t - \frac{Q}{R}
    }} \\
    &= \arccos\lrp{ \frac{1}{4\sqrt{Q^2+PR}} \lrb{ 
        -2Pb_0 \tfrac{R}{S} - 4Q - 2\frac{R^2}{S^2} t 
    }} \\
    &= \arccos\lrp{ \frac{ -B - 2A t }{\sqrt D} }.
\end{align}

Finally, we know from Theorem \ref{thm:equid-W--A<0,D>0} that $W_{\Delta}^{(-A,-B,-C)}$ is $d\mu = \frac{1}{-At^2-Bt-C} dt$ equidistributed over $(\frac{-B-\sqrt{D}}{2A},~ \frac{-B+\sqrt{D}}{2A})$. And under the reparametrization 
$\theta(t) = \arccos\lrp{ \frac{ -B - 2A t }{\sqrt D} }$, this measure 
$d\mu = \frac{1}{-At^2-Bt-C} dt$ over $t \in (\frac{-B-\sqrt{D}}{2A},~ \frac{-B+\sqrt{D}}{2A})$ becomes $d\mu = \frac{2}{\sqrt{D}} \frac{1}{\sin \theta} d\theta$ over $\theta \in (0,\pi)$. Thus we have that $\arg_q \CM^{G}_{\Delta} = \theta(W_{\Delta}^{(-A,-B,-C)})$ is equidistributed over $(0,\pi)$ with respect to the hyperbolic metric $\frac{1}{\sin \theta} d\theta$, completing the proof.
\end{proof}

We remark here that Case 1 of the above proof (for half-line geodesics) could also be shown as a corollary of Case 2 (for semicircle geodesics); see the paragraph below for explanation. We have chosen here to prove both parts directly for sake of exposition. The argument for Case 1 turns out to be much easier, so we gave that argument first. Then the argument for Case 2 generalizes the ideas from Case 1.

The reason that Case 1 follows from Case 2 is that every half-line rational geodesic can be written as an $\SL_2(\ZZ)$ transformation of a semicircle rational geodesic.
Given a half-line rational geodesic $G_x$, choose a natural number $n > x$ and let $q = r = \frac{1}{2n-2x}$.  
Then it is straightforward to verify that the $\SL_2(\ZZ)$ transformation $\gamma = \pttmatrix{n}{-1}{1}{0} : z \mapsto \frac{nz-1}{z} $ bijectively maps 
\begin{align}
    \gamma ~:~ G_{q,r} =\{q + re^{i\theta} ~:~ \theta \in (0,\pi)\} &\longrightarrow 
    G_x =\{x+iy ~:~ y \in (0,\infty)\}.
\end{align}
Under this mapping, we have that $\theta \mapsto y(\theta) =\frac{1}{2r} \frac{\sin \theta}{1+\cos \theta}$, and so the metric $ds_\text{hyp} = \frac{1}{\sin \theta} d\theta$ becomes $ds_\text{hyp} = \frac{1}{y}dy$.
As an $\SL_2(\ZZ)$ transformation, $\gamma$ preserves discriminants of CM points.
Hence applying $\gamma$ to $\CM^{G_{q,r}}_{|D|\le \Delta}$, we have by Case 2 that $\textnormal{Im\,} \CM^{G_x}_{|D|\le \Delta}$ is equidistributed with respect to $\frac{1}{y}dy$, verifying Case 1 for $G_x$.

%%%%%%%%%%%%%%%%%%%%%%%%%%%%%%%%%%

{
\renewcommand{\thetheorem}{\ref{thm:equid-RM-geodesic}}
\begin{theorem} 
    Fix a rational geodesic $G$, and let 
    $\RM^{G}_{\Delta}$ denote the set of RM curves along $G$ with discriminant $D \le \Delta$. 
    Then as $\Delta \to \infty$, $\RM^{G}_{\Delta}$ is equidistributed along $G$ with respect to the hyperbolic metric.
\end{theorem}
\addtocounter{theorem}{-1}
}
\begin{proof}  Let $\RM^{\cap G}_{\Delta}$ denote the set of intersection points of the RM curves in $\RM^{G}_{\Delta}$ with $G$. We want to show that as $\Delta\to\infty$, $\RM^{\cap G}_{\Delta}$ is equidistributed along $G$.

\noindent\textbf{Case 1: $G$ is a half-line rational geodesic.}

Write $G$ as $G_{(0,B_0,C_0)}$. Then by an identical argument as for \eqref{eqn:CM-G0BC-set-formula} in Theorem \ref{thm:equid-CM-geodesic} (and using the same parameters), we have that
\begin{align}
    \RM^{G}_{\Delta} 
    &= \Big\{  \langle nQ,nP,-m \rangle  ~:~ n \ge 1,~ \gcd(m,n)=1,~  0 < B mn + C n^2 \le \Delta
    \Big\},
\end{align}
and so
\begin{align}
    \mathrm{Im}\,\RM^{\cap G}_{\Delta} 
    &= \Big\{  \frac{\sqrt{4Qnm +P^2n^2}}{2Qn}  ~:~ n \ge 1,~ \gcd(m,n)=1,~  0 < B mn + C n^2 \le \Delta
    \Big\} \\
    &= \Big\{  \sqrt{\frac{4}{B} \lrp{\frac mn} + \frac{4C}{B^2}}  ~:~ n \ge 1,~ \gcd(m,n)=1,~  0 < B mn + C n^2 \le \Delta
    \Big\} \\
    &=  \Big\{  y\lrp{\frac mn}  ~:~ n \ge 1,~ \gcd(m,n)=1,~  0 < B mn + C n^2 \le \Delta
    \Big\} \\
    &= y(W_{\Delta}^{(0,B,C)}), \\
    \text{where }\ y(t) 
    &:= \sqrt{\frac{4}{B} t + \frac{4C}{B^2}}.
\end{align}

Finally, we know from Theorem \ref{thm:equid-W--A=0,B>0} that $W_{\Delta}^{(0,B,C)}$ is $d\mu = \frac{1}{Bt+C}$ equidistributed over $(\frac{-C}{B},\infty)$. And under the reparametrization 
$y(t) = \sqrt{\frac{4}{B} t + \frac{4C}{B^2}}$, this measure $d\mu = \frac{1}{Bt+C}$ over $t \in (\frac{-C}{B},\infty)$ becomes $d\mu = \frac{2}{B} \frac{1}{y} dy$ over $y \in (0,\infty)$. Thus we have that $\mathrm{Im}\,\RM^{\cap G}_{\Delta} = y(W_{\Delta}^{(0,B,C)})$ is equidistributed over $(0,\infty)$ with respect to the hyperbolic metric $\frac{1}{y} dy$, completing the proof.

\noindent\textbf{Case 2: $G$ is a semicircle rational geodesic.}

Write $G$ as $G_{(A_0,B_0,C_0)}$. Then by an identical argument as for \eqref{eqn:CM-GABC-set-formula} in Theorem \ref{thm:equid-CM-geodesic} (and using the same parameters), we have that
\begin{align}
    \RM^{G}_{\Delta} 
    &= \Big\{\langle nS,~nPb_0+m\tfrac{R}{S},~nPc_0-m\tfrac{Q}{S} \rangle ~:~ n\ge 1,~ \gcd(m,n)=1,~ \\
    &\hspace{61mm} 0 < A m^2 + B mn + C n^2  \le \Delta
    \Big\}, \\
    \text{where}\quad  A&:=\tfrac{R^2}{S^2}, \qquad B := 2Pb_0 \tfrac{R}{S}+4Q, \qquad C := P^2b_0^2-4SPc_0,  \\
     D &:= B^2 -4A C = 16Q^2+16PR = \frac{16}{\gcd(D_0,2)^2} D_0 > 0.
\end{align}

Now, let $z$ denote the perpendicular-intersection point of $\langle a,b,c\rangle$ and $G=G_{(A_0,B_0,C_0)}$. By \eqref{eqn:GABC-set-formula}, we know that $z$ satisfies
\begin{align}
    A_0|z|^2 + B_0 \, \mathrm{Re}(z) + C_0 = 0 \qquad \text{and}\qquad a|z|^2 + b \, \mathrm{Re}(z) + c = 0.
\end{align}
Thus
\begin{align}
    &\lrp{\frac{B_0}{A_0} - \frac ba} \mathrm{Re}(z) + \lrp{\frac{C_0}{A_0} - \frac ca} = 0, \\
    \text{and so }\ \ \mathrm{Re}(z) &= \frac{\frac ca -\frac{C_0}{A_0}}{\frac{B_0}{A_0} - \frac ba} = \frac{A_0c-C_0a}{B_0a-A_0b} = \frac{B_0b-4C_0a}{2B_0a-2A_0b} \qquad\quad \text{(by Lemma \ref{lem:CM-RM-along-G-diophantine})} \\
    &= \frac{-Qb+2Pa}{-2Qa-Rb} = \frac{\frac{Q}{R}\lrp{-2Qa-Rb} + 2\frac{Q^2}{R}a + 2Pa}{-2Qa-Rb} \quad\quad \text{(by \eqref{eqn:def-PQR})}\\
    &= \frac{Q}{R} -  \frac{\frac{2a}{R}(Q^2+PR)}{2Qa+Rb}.
\end{align}
Hence for an angle $\theta=\arg_q \langle nS,~nPb_0+m\tfrac{R}{S},~nPc_0-m\tfrac{Q}{S} \rangle \cap G$ in the set $\mathrm{arg}_q\, \RM^{\cap G}_{\Delta}$, we have that 
\begin{align}
    q+r \cos\theta = \frac{Q}{R} -  \frac{\frac{2nS}{R}(Q^2+PR)}{2QnS+R\lrp{nPb_0+m\tfrac{R}{S}}} = 
    \frac{Q}{R} - \frac{\frac{2S}{R}(Q^2+PR)}{\frac{R^2}{S}\lrp{\frac mn} +2QS+RPb_0  },
\end{align}
and so
\begin{align}
    \arg_q \RM^{\cap G}_{\Delta} 
    &= \Big\{\theta \lrp{\frac{m}{n}} ~:~ n\ge 1,~ \gcd(m,n)=1,~   0  < A \lrp{\frac{m}{n}}^2 + B \lrp{\frac{m}{n}} + C  \le \frac{\Delta}{n^2}
    \Big\} \\
    &= \theta(W_{\Delta}^{(A,B,C)}), \\
    \text{where }\ \theta(t) &:= \arccos\lrp{ \frac 1 r \lrb{
        \frac{Q}{R} - \frac{\frac{2S}{R}(Q^2+PR)}{\frac{R^2}{S}t +2QS+RPb_0  }
    - q}}.
\end{align}

Also note that
\begin{align} 
    q = \frac{Q}{R}, &\qquad r^2-q^2 = \frac{P}{R}, \qquad r=\frac{\sqrt{Q^2+PR}}{R},  \\
    \text{so }\ 
    \theta(t) 
    &= \arccos\lrp{ \frac{R}{\sqrt{Q^2+PR}} \lrb{
        \frac{Q}{R} - \frac{\frac{2S}{R}(Q^2+PR)}{\frac{R^2}{S}t +2QS+RPb_0  }
    - \frac QR }} \\
    &= \arccos\lrp{ \frac{1}{4\sqrt{Q^2+PR}} \lrb{
         - \frac{16(Q^2+PR)}{2\frac{R^2}{S^2}t +4Q+\frac{R}{S}Pb_0  }
    }} \\
    &= \arccos\lrp{ \frac{-\sqrt D}{ 2A t+B } }.
\end{align}

Finally, we know from Theorem \ref{thm:equid-W--A>0,D>0} that $W_{\Delta}^{(A,B,C)}$ is $d\mu = \frac{1}{At^2+Bt+C} dt$ equidistributed over  $(\frac{-B+\sqrt{D}}{2A},\infty] \cup [-\infty,\frac{-B-\sqrt{D}}{2A})$. And under the reparametrization 
$\theta(t) = \arccos\lrp{ \frac{-\sqrt D}{ 2A t+B } }$, this measure
$d\mu = \frac{1}{At^2+Bt+C} dt$ over $t \in (\frac{-B+\sqrt{D}}{2A},\infty] \cup [-\infty,\frac{-B-\sqrt{D}}{2A})$ becomes $\frac{2}{\sqrt{D}} \frac{1}{\sin \theta} d\theta$ over $\theta \in (0,\pi)$. Thus we have that $\arg_q \RM^{\cap G}_{\Delta} = \theta(W_{\Delta}^{(A,B,C)})$ is equidistributed over $(0,\pi)$ with respect to the hyperbolic metric $\frac{1}{\sin \theta} d\theta$, completing the proof.
\end{proof}

{
\renewcommand{\thetheorem}{\ref{thm:equid-RM-fixedpt}}
\begin{theorem} 
    Fix a CM point $p$, and let 
    $\RM^p_{\Delta}$ denote the set of RM curves passing through $p$ with discriminant $D \le \Delta$. 
    Then as $\Delta \to \infty$, the angles $\mathrm{ang}_p\, \RM^p_{\Delta}$ are equidistributed with respect to the uniform angle measure.
\end{theorem}
\addtocounter{theorem}{-1}
}
\begin{proof}  
Write $p$ as $[A_0,B_0,C_0]$. Then by an identical argument as for \eqref{eqn:CM-GABC-set-formula} in Theorem \ref{thm:equid-CM-geodesic} (and using the same parameters), we have that
\begin{align}
    \RM^{p}_{\Delta} 
    &= \Big\{\langle{nS,~nPb_0+m\tfrac{R}{S},~nPc_0-m\tfrac{Q}{S}}\rangle  ~:~ n\ge 1,~ \gcd(m,n)=1,~ \\
    &\hspace{41mm} 0 < A m^2 + B mn + C n^2  \le \Delta
    \Big\} \\
    \text{where}\quad  A&:=\tfrac{R^2}{S^2}, \qquad B := 2Pb_0 \tfrac{R}{S}+4Q, \qquad C := P^2b_0^2-4SPc_0,  \\
     D &:= B^2-4AC = 16Q^2+16PR = \frac{16}{\gcd(D_0,2)^2} D_0 < 0.
\end{align}

Now, for an angle $\theta = \mathrm{ang}_p\, \langle a,b,c \rangle$, note that
\begin{align}
    \frac{Q}{R} &= \frac{-B_0}{2A_0} = \mathrm{Re}(p) = \frac{-b}{2a} + \frac{\sqrt{b^2-4ac}}{2a} \cos \theta \\
    &\text{and so }\ \ \cos \theta = \frac{2a}{\sqrt{b^2-4ac}}\lrp{\frac{Q}{R} + \frac{b}{2a}}.
\end{align}
Hence for an angle $\theta=\mathrm{ang}_p\, \langle{nS,~nPb_0+m\frac{R}{S},~nPc_0-m\frac{Q}{S}}\rangle$ in the set $\mathrm{ang}_p\, \RM^{p}_{\Delta}$, we have that
\begin{align}
    \cos \theta 
    &= \frac{2nS}{\sqrt{(nPb_0+m\frac{R}{S})^2-4nS(nPc_0-m\frac{Q}{S})}}\lrp{\frac{Q}{R} + \frac{nPb_0+m\frac{R}{S}}{2nS}} \\
    &= \frac{2nS}{n\sqrt{A \lrp{\frac mn}^2 + B \lrp{\frac mn} + C}}\lrp{\frac{4Q+2Pb_0 \frac RS}{4R} + \frac{R}{2S^2} \lrp{\frac mn}} \\
    &= \frac{2S}{\sqrt{A \lrp{\frac mn}^2 + B \lrp{\frac mn} + C}}\lrp{\frac{B+2A\lrp{\frac mn}}{4R}} \\
    &= \frac{1}{2\sqrt A}\frac{B+2A\lrp{\frac mn}}{\sqrt{A \lrp{\frac mn}^2 + B \lrp{\frac mn} + C}},
\end{align}
and so
\begin{align}
    \mathrm{ang}_p\, \RM^{p}_{\Delta} 
    &= \Big\{\theta\!\lrp{\frac{m}{n}} ~:~ n\ge 1,~ \gcd(m,n)=1,~   0  < A m^2 + B mn + C n^2  \le \Delta
    \Big\} \\
    &= \theta(W_{\Delta}^{(A,B,C)}), \\
    \text{where }\ \theta(t) &:= \arccos\lrp{\frac{1}{2\sqrt A}\frac{B+2At}{\sqrt{A t^2 + Bt + C}}}.
\end{align}

Finally, we know from Theorem \ref{thm:equid-W--A>0,D<0} that $W_{\Delta}^{(A,B,C)}$ is $d\mu = \frac{1}{At^2+Bt+C} dt$ equidistributed over  $(-\infty,\infty)$. And under the reparametrization $\theta(t) = \arccos\lrp{\frac{1}{2\sqrt A}\frac{B+2At}{\sqrt{A t^2 + Bt + C}}}$, this measure
$d\mu = \frac{1}{At^2+Bt+C} dt$ over $t \in [-\infty,\infty]$ becomes $d\mu = \frac{2}{\sqrt{-D}}  d\theta$ over $\theta \in [0,\pi]$. Thus we have that $\mathrm{ang}_p\, \RM^{p}_{\Delta} = \theta(W_{\Delta}^{(A,B,C)})$ is equidistributed over $[0,\pi]$ with respect to the uniform angle measure $d\theta$.
\end{proof}

%%%%%%%%%%%%%%%%%%%%%%%%%%%%%%%%%%%%%%%%%%
%%%%%%%%%%%%%%%%%%%%%%%%%%%%%%%%%%%%%%%%%%

\section{Proof of Theorem \ref{thm:equid-CM-fixedpt}} \label{sec:proof-of-thm-equid-CM-fixedpt}

To prove Theorem \ref{thm:equid-CM-fixedpt}, we first need to carry out some hyperbolic geometry calculations.
\begin{lemma} 
    \label{lem:hyp-dist-and-angle}
    Let $x_0 + iy_0$ and $x+iy$ be two points in $\HH$. Then the hyperbolic distance from $x_0+iy_0$ to $x+iy$ is given by
    \begin{align} \label{eqn:geom-lemma--dist}
        \mathrm{dist}_\mathrm{hyp}(x+iy, x_0+iy_0) = \mathrm{arccosh}\lrp{\frac{(x-x_0)^2 + y^2 + y_0^2}{2yy_0}}.
    \end{align}
    Moreover, if $x > x_0$, then the angle from $x_0 + iy_0$ to $x+iy$ is given by
    \begin{align} 
        &\mathrm{ang}_{x_0+iy_0}(x+iy) = \arccos \lrp{\frac{q-x_0}{r}}, 
        \label{eqn:geom-lemma--angle}
        \\
        \text{where}&\quad q := \frac{x+x_0}{2} + \frac{y^2-y_0^2}{2(x-x_0)} \quad \text{and} \quad r := \sqrt{(x_0-q)^2+y_0^2}.
        \label{eqn:geom-lemma--q-r}
    \end{align}
\end{lemma}
\begin{proof}
    The distance formula \eqref{eqn:geom-lemma--dist} is well known. Here, we prove \eqref{eqn:geom-lemma--angle}.

    Let $G=G_{q,r}$ denote the geodesic connecting $x_0+iy_0$ and $x_iy$. Then 
    \begin{align}
        (x_0-q)^2 + y_0^2 = r^2 \qquad \text{and} \qquad (x-q)^2+y^2 = r^2
    \end{align}
    Solving this system for $q$ and $r$ yields the claimed formulas \eqref{eqn:geom-lemma--q-r}. 
    
    Now, since $x > x_0$, recall that $\mathrm{ang}_{x_0+iy_0}(x+iy) = \mathrm{ang}_{x_0+iy_0}(G)$, and that $\mathrm{ang}_{x_0+iy_0}(G)$ can also be interpreted as the argument of $x_0+iy_0$ as a point on $G$ (see Figure \ref{fig:geodesic-angle}). Hence
    \begin{align}
        \theta &= \mathrm{ang}_{x_0+iy_0}(x+iy) = \mathrm{ang}_{x_0+iy_0}(G) = \mathrm{arg}_G(x_0+iy_0), \\
        &\text{so }\text{that}\quad x_0 = \text{Re}(q+r \,e^{i\theta}) = q + r \cos(\theta).
    \end{align}
    This then implies \eqref{eqn:geom-lemma--angle}, completing the proof.
\end{proof}

Now, Theorem \ref{thm:equid-CM-fixedpt} concerns the behavior of CM points in the (hyperbolic) closed ball $\overline B(z_0,s_0)$. Writing $z_0 = x_0 +i y_0$ and using the distance formula \eqref{eqn:geom-lemma--dist}, it is straightforward to compute that $\overline B(z_0,s_0)$ is the Euclidean disk in $\HH$ with center $x_0 + i y_0 \cosh(s_0)$ and Euclidean radius $y_0 \sinh(s_0)$. In the following lemma, we determine the area of the region in $\overline B(z_0,s_0)$ with angles $\theta_1 \le \mathrm{ang}_{z_0}(z) \le \theta_2$.

\par \raggedbottom
\begin{figure}[H]
    \centering
    \includegraphics[width=0.6\textwidth]{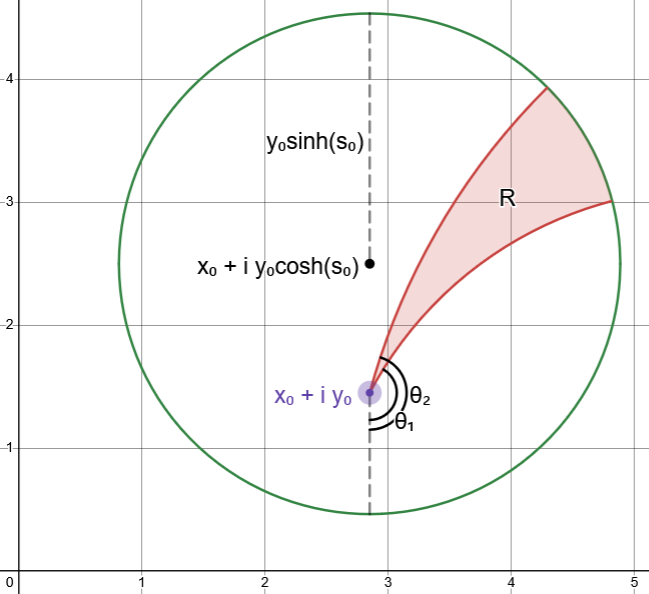}
    \caption{The region $R = \lrcb{z \in \overline B(x_0+iy_0,s_0) ~:~ \theta_1 \le \mathrm{ang}_{x_0+iy_0}(z) \le \theta_2}$.}
    \label{fig:hyperbolic-disk}
\end{figure}
\par \flushbottom 

\begin{lemma} \label{lem:area-of-R}
    Fix $z_0 \in \HH$, $s_0 \in \RR^+$, and $0 \le \theta_1 < \theta_2 \le 2\pi$. Let $\overline B(z_0,s_0)$ denote the set of all points in $\HH$ within hyperbolic distance $s_0$ of $z_0$. Then the hyperbolic area of the region
    \begin{align}
        R := \lrcb{z \in \overline B(z_0,s_0) ~:~ \theta_1 \le \mathrm{ang}_{z_0}(z) \le \theta_2
        }
    \end{align}
    is $(\theta_2-\theta_1)(\cosh(s_0)-1)$.
\end{lemma}
\begin{proof}
    We can assume without loss of generality that $0 < \theta_1 < \theta_2 < \pi$, so that $R$ lies entirely in the right half of $\overline B_{z_0}(s_0)$.
    Then we reparametrize the points $x+iy$ in $R$ by $\theta = \mathrm{ang}_{z_0}(x+iy)$ and $s = \mathrm{dist}_\mathrm{hyp}(x+iy,z_0)$. Writing $z_0=x_0+iy_0$, we have by Lemma \ref{lem:hyp-dist-and-angle} that
     \begin{align}
         s(x,y) &= \mathrm{arccosh}\lrp{\frac{(x-x_0)^2 + y^2 + y_0^2}{2yy_0}} \quad
         \text{and}\quad\theta(x,y) = \arccos \lrp{\frac{q(x,y)-x_0}{r(x,y)}}, 
        \\
        \text{where}&\quad q(x,y) := \frac{x+x_0}{2} + \frac{y^2-y_0^2}{2(x-x_0)} \quad \text{and} \quad r(x,y) := \sqrt{(x_0-q(x,y))^2+y_0^2}.
     \end{align}

     It is straightforward (though somewhat tedious) to calculate that the Jacobian of this reparametrization is $J = s_x \theta_y - s_y \theta_x  = \frac{-1}{y^2 \sinh( s(x,y))}$. This then yields
     \begin{align}
         \mathrm{area}_\mathrm{hyp}(R) 
         &= \iint_R \frac{dx\,dy}{y^2} = \iint_R \sinh(s(x,y)) \,|J|\, dx \,dy = \int_{\theta_1}^{\theta_2} \int_{0}^{s_0} \sinh(s)\, ds \,d\theta \\
         &= \int_{\theta_1}^{\theta_2} \cosh(s_0)-1 \,d\theta = (\theta_2-\theta_1) (\cosh(s_0)-1),
     \end{align}
     proving the desired result.    
\end{proof}

Using this lemma, we are now able to prove Theorem \ref{thm:equid-CM-fixedpt}.
{
\renewcommand{\thetheorem}{\ref{thm:equid-CM-fixedpt}}
\begin{theorem} 
    Fix $z_0 \in \HH$ and $s_0 \in \RR^+$, and for negative fundamental discriminants $D$, let 
    $\CM^{\overline B(z_0, s_0)}_{D}$ denote the set of CM points of discriminant $D$
    within hyperbolic distance $s_0$ of $z_0$.
    Then as $D \to -\infty$, the angles $\mathrm{ang}_{z_0}\, \CM^{\overline B(z_0,s_0)}_{D}$ are equidistributed with respect to the uniform angle measure.
\end{theorem}
\addtocounter{theorem}{-1}
}
\begin{proof}
    Recall that in \cite[Theorem 1 (i)]{duke}, Duke showed that as fundamental discriminants $D \to -\infty$, CM points of discriminant $D$ are equidistributed over $\HH$ with respect to the hyperbolic measure. This means that for any fixed $[X,Y] \subseteq [0,2\pi]$,
    \begin{align}
        \lim_{\substack{D \to -\infty \\ \text{fund. disc.}}}
        \frac{
            \#\, \mathrm{ang}_{z_0}\CM^{\overline B(z_0, s_0)}_{D} \cap [X,Y]
        }{
            \#\, \mathrm{ang}_{z_0}\CM^{\overline B(z_0, s_0)}_{D}
        } 
        &= 
        \frac{
            \mathrm{area}_\mathrm{hyp}\lrp{\lrcb{z \in \overline B(z_0,s_0) ~:~ X \le \mathrm{ang}_{z_0}(z) \le Y}}
        }{
            \mathrm{area}_\mathrm{hyp}\lrp{\overline B(z_0,s_0)}
        } \\
        &= 
        \frac{
            (Y-X)(\cosh(s_0)-1)
        }{
            (2\pi-0)(\cosh(s_0)-1)
        } \qquad\qquad \text{(by Lemma \ref{lem:area-of-R})} \\
        &= 
        \frac{
            \int_X^Y d\theta
        }{
            \int_0^{2\pi} d\theta
        },
    \end{align}
    verifying the desired result.
\end{proof}

%%%%%%%%%%%%%%%%%%%%%%%%%%%%%%%%%%%
%%%%%%%%%%%%%%%%%%%%%%%%%%%%%%%%%%%

\section{Analytic Number Theory Lemmas} 
\label{sec:equid-WABC--ANT-lemmas}

In Sections \ref{sec:equid-WABC--ANT-lemmas} - \ref{sec:equid-WABC--A<0}, we carry out all of the technical calculations to prove the equidistribution result of Theorem \ref{thm:equid-WABC}. 
Since the proof of this theorem is purely just an equidistribution calculation, these four sections are more or less independent from the rest of the paper. In particular, in these four sections, we will not make any reference to geodesics, CM points, RM curves, or any other objects on $\HH$.

In this section, we give six lemmas; all of which follow from standard analytic number theory arguments. 
In the following, $n,m,d$ will always denote natural numbers. Addtionally, $\phi$ denotes the Euler totient function, $\mu$ denotes the Mobius function, and $\gamma$ denotes the Euler-Mascheroni constant.
\begin{lemma}[{\cite[Theorem 8.29]{niven-zuckerman-montgomery}}] \label{lem:count-coprime-range}
    Fix $\varepsilon > 0$. Then as $T$ tends to infinity, 
    \begin{align}
        \#\{m \le T ~:~ \gcd(m,n)=1\} = \frac{\phi(n)}{n} T + O(n^{\varepsilon}) \quad \text{uniformly over all $n \ge 1$}.
    \end{align}
\end{lemma}

\begin{lemma}[{\cite[Section 8.3, Problem 19]{niven-zuckerman-montgomery}}] \label{lem:sum-phi}
    As $T$ tends to infinity, 
    \begin{align}
        \sum_{n \le T} \phi(n) &= \frac{3}{\pi^2} T^2 + O(T \log T).
    \end{align}
\end{lemma}

\begin{lemma} \label{lem:sum-phi-over-a}
    As $T$ tends to infinity, 
    \begin{align}
        \sum_{n \le T} \frac{\phi(n)}{n} &= \frac{6}{\pi^2} T + O\lrp{\log T}.
    \end{align}
\end{lemma}
\begin{proof}
    We have 
    \begin{align}
        \sum_{n \le T} \frac{\phi(n)}{n} &= \sum_{n \le T} \sum_{d|n}  \frac{\mu(d)}{d} = \sum_{d \le T} \lrfloor{\frac Td}\frac{\mu(d)}{d}  = \sum_{d \le T} \frac{\mu(d)}{d} \lrp{\frac Td + O\lrp{1} }
        \\
        &= T\sum_{d \le T} \frac{\mu(d)}{d^2} +  O\lrp{\log T}
        = T\lrp{\frac{6}{\pi^2} + O\lrp{\frac1T}}  + O\lrp{\log T} \\
        &= \frac{6}{\pi^2} T  + O\lrp{\log T},
    \end{align}
    as desired.
\end{proof}

\begin{lemma} \label{lem:sum-phi-over-a2}
    As $T$ tends to infinity, 
    \begin{align}
        \sum_{n \le T} \frac{\phi(n)}{n^2} &= \frac{6}{\pi^2} \log T + \gamma_0 + O\lrp{\frac{\log T}{T}} 
    \end{align}
    for a certain constant $\gamma_0$.
\end{lemma}
\begin{proof}
    We have 
    \begin{align}
        \sum_{n \le T} \frac{\phi(n)}{n^2} &= \sum_{n \le T} \frac{1}{n} \sum_{d|n} \frac{\mu(d)}{d} = \sum_{d \le T} \frac{\mu(d)}{d} \sum_{\substack{n \le T \\ d|n}} \frac{1}{n} = \sum_{d \le T} \frac{\mu(d)}{d^2} \sum_{n' \le T/d } \frac{1}{n'}  \\
        &= \sum_{d \le T} \frac{\mu(d)}{d^2} \lrp{\log (T/d) + \gamma + O\lrp{\frac{1}{T/d}} }
        \\
        &= (\log T) \sum_{d \le T} \frac{\mu(d)}{d^2} + \sum_{d \le T} \frac{\mu(d)(\gamma-\log d)}{d^2} + \sum_{d \le T} \frac{\mu(d)}{d} O\lrp{\frac1T}
        \\
        &= (\log T)\lrp{\frac{6}{\pi^2} + O\lrp{\frac1T}} + \lrp{\gamma_0 + O\lrp{\frac{\log T}{T}}} + O\lrp{\frac{\log T}{T}} \\
        &= \frac{6}{\pi^2} \log T + \gamma_0 + O\lrp{\frac{\log T}{T}},
    \end{align}
    as desired.
    Note that $\gamma_0$ here is the constant
    $\displaystyle
        \gamma_0 := \sum_{d \ge 1} \frac{\mu(d)(\gamma-\log d)}{d^2}.
    $
\end{proof}

\begin{lemma}
    \label{lem:sum-phi-sqrt}
    Fix $A \ne 0$, $D > 0$, and $s_2 \ge s_1 \ge \max(0, \frac{-4A}{D})$. Then as $\Delta$ tends to infinity,
    \begin{align}
        &\sum_{ \sqrt{s_1 \Delta} < n \le \sqrt{s_2 \Delta} }
        \phi(n)  
        \sqrt{D+\frac{4A\Delta}{n^2}} \\
        &= \frac{3\Delta}{\pi^2}\Bigg[ \sqrt{Du^2+4Au} +  \frac{4A}{\sqrt D}
            \log\!\lrp{
                \sqrt{Du} + \sqrt{Du+4A}
            }
         \Bigg]_{u=s_1}^{u=s_2} + O\!\lrp{\sqrt{\Delta} \log^2 \Delta}.
    \end{align}
\end{lemma}
\begin{proof}  
    We focus on the case of $s_1 > 0$. The only modification needed for the case of $s_1 = 0$ is that the expression $x^2 + O(x\log x)$ in \eqref{eqn:temp-integral-Oxlogx} below should be understood to be identically $0$ for $0\le x \le 1$.
    Denote 
    \begin{align}
        \Phi(x) := \sum_{n \le x} \phi(n) = \frac{3}{\pi^2}x^2 + O(x\log x) \qquad \text{(by Lemma \ref{lem:sum-phi})},
    \end{align}
    and also observe that
    \begin{align}
        \frac{d}{dx}\lrb{\sqrt{D+\frac{4A\Delta}{x^2}}} =  \frac{-4A\Delta}{x^2\sqrt{Dx^2+4A\Delta}}.
    \end{align}
    Then we have by Abel's summation formula that
    \begin{align}
        &\sum_{\sqrt{s_1 \Delta} < n \le \sqrt{s_2 \Delta}} \phi(n)\sqrt{D+\frac{4A\Delta}{n^2}} \\
        &= \lrb{\Phi(x)\sqrt{D+\frac{4A\Delta}{x^2}}}_{x=\sqrt{s_1 \Delta}}^{x=\sqrt{s_2 \Delta}} - \int_{\sqrt{s_1 \Delta}}^{\sqrt{s_2 \Delta}} \Phi(x) 
        \frac{-4A\Delta}{x^2\sqrt{Dx^2+4A\Delta}}
        \,dx \\
        &= \frac{3}{\pi^2}\lrb{(x^2+O(x \log x))\sqrt{D+\frac{4A\Delta}{x^2}} }_{x=\sqrt{s_1 \Delta}}^{x=\sqrt{s_2 \Delta}} 
         + \frac{3}{\pi^2} \int_{\sqrt{s_1 \Delta}}^{\sqrt{s_2 \Delta}}  \frac{x^2 + O(x \log x)}{x^2} \frac{4A\Delta}{\sqrt{Dx^2+4A\Delta}} \,dx 
         \label{eqn:temp-integral-Oxlogx}\\
        &= \frac{3}{\pi^2}\lrb{\sqrt{Dx^4+4Ax^2}}_{x=\sqrt{s_1\Delta}}^{x=\sqrt{s_2\Delta}}  
         + \frac{3}{\pi^2} \int_{\sqrt{s_1 \Delta}}^{\sqrt{s_2 \Delta}}   \frac{4A\Delta}{\sqrt{Dx^2+4A\Delta}} \,dx + O\!\lrp{\sqrt{\Delta} \log^2 \Delta} \label{eqn:temp-integral-lemma-sum-phi-sqrt}\\
        &= \frac{3}{\pi^2}\lrb{\Delta \sqrt{Du^2+4Au}}_{u=s_1}^{u=s_2} 
         + \frac{3}{\pi^2 } \lrb{\frac{4A\Delta}{\sqrt{D}} 
         \log\!\lrp{
                \sqrt{Dx^2} + \sqrt{Dx^2+4A\Delta}
            }
         }_{x=\sqrt{s_1 \Delta}}^{x=\sqrt{s_1 \Delta}} + O\!\lrp{\sqrt{\Delta} \log^2 \Delta} \\
         &= \frac{3\Delta}{\pi^2}\Bigg[  \sqrt{Du^2+4Au} +  \frac{4A}{\sqrt D}
            \log\!\lrp{
                \sqrt{Du} + \sqrt{Du+4A}
            }
         \Bigg]_{u=s_1}^{u=s_2} + O\!\lrp{\sqrt{\Delta} \log^2 \Delta},
    \end{align}
    as desired.
\end{proof}

%%%%%%%%%%%

\begin{lemma}
    \label{lem:sum-phi-sqrt-neg-discr}
    Fix $A > 0$, $D < 0$, and $0 \le s_1 \le s_2 \le \frac{4A}{-D}$. Then as $\Delta$ tends to infinity,
    \begin{align}
        &\sum_{ \sqrt{s_1 \Delta} < n \le \sqrt{s_2 \Delta} }
        \phi(n)  
        \sqrt{D+\frac{4A\Delta}{n^2}} \\
        &= \frac{3\Delta}{\pi^2}\Bigg[  \sqrt{Du^2+4Au} -  \frac{4A}{\sqrt{-D}}
            \arctan\!\lrp{\frac{\sqrt{Du+4A}}{\sqrt{-Du}}}
         \Bigg]_{u=s_1}^{u=s_2} + O\!\lrp{\sqrt{\Delta} \log^2 \Delta}.
    \end{align}
\end{lemma}
\begin{proof}  
    By an identical argument as in \eqref{eqn:temp-integral-lemma-sum-phi-sqrt}, we have that
    \begin{align}
        &\sum_{\sqrt{s_1 \Delta} < n \le \sqrt{s_2 \Delta}} \phi(n)\sqrt{D+\frac{4A\Delta}{n^2}} \\
        &= \frac{3}{\pi^2}\lrb{\sqrt{Dx^4+4Ax^2}}_{x=\sqrt{s_1\Delta}}^{x=\sqrt{s_2\Delta}}  
         + \frac{3}{\pi^2} \int_{\sqrt{s_1 \Delta}}^{\sqrt{s_2 \Delta}}   \frac{4A\Delta}{\sqrt{Dx^2+4A\Delta}} \,dx + O\!\lrp{\sqrt{\Delta} \log^2 \Delta} \\
        &= \frac{3}{\pi^2}\lrb{\Delta \sqrt{Du^2+4Au}}_{u=s_1}^{u=s_2} 
         + \frac{3}{\pi^2 } \lrb{\frac{-4A\Delta}{\sqrt{-D}} 
         \arctan\!\lrp{
            \frac{\sqrt{Dx^2+4A\Delta}}{\sqrt{-Dx^2}}
         }
         }_{x=\sqrt{s_1 \Delta}}^{x=\sqrt{s_1 \Delta}} + O\!\lrp{\sqrt{\Delta} \log^2 \Delta} \\
         &= \frac{3\Delta}{\pi^2}\Bigg[  \sqrt{Du^2+4Au} -  \frac{4A}{\sqrt{-D}}
            \arctan\!\lrp{
                \frac{\sqrt{Du+4A}}{\sqrt{-Du}}
             }
         \Bigg]_{u=s_1}^{u=s_2} + O\!\lrp{\sqrt{\Delta} \log^2 \Delta},
    \end{align}
    as desired.
\end{proof}

%%%%%%%%%%%%%%%%%%%%%%%%%%%%%%%%%%%
%%%%%%%%%%%%%%%%%%%%%%%%%%%%%%%%%%%%

\section{Proof of Theorem \ref{thm:equid-WABC} for \texorpdfstring{$A=0$}{A=0}}
\label{sec:equid-WABC--A=0}

{
\renewcommand{\thetheorem}{\ref{thm:equid-WABC} (\(A=0,B>0\))}
\begin{theorem} 
    \label{thm:equid-W--A=0,B>0}
    Fix $(B,C) \in \RR^2$ with $B > 0$. Then
    \begin{align}
        W_{\Delta}^{(0,B,C)} = \lrcb{\frac{m}{n}  ~:~ n\ge1,~ \gcd(m,n)=1,~ 0 < B mn + C n^2 \le \Delta} 
    \end{align} 
    is equidistributed over $(\frac{-C}{B},\infty)$ with respect to the measure $d\mu = \frac{1}{Bt+C} \,dt$.
\end{theorem}
}

\begin{proof}
    % Fix $[K_1,\infty) \subseteq (\frac{-C}{B},\infty)$ and $\varepsilon > 0$. In particular, 
    To prove equidistribution, we show that for each $[X,Y] \subseteq (\frac{-C}{B},\infty)$,
    \begin{align} \label{eqn:goal-equid-A=0,B>0} \tag{$*$}
        \#\, W_{\Delta}^{(0,B,C)} \cap [X,Y] = 
        \frac{3\Delta}{\pi^2} \lrb{ 
        \frac1B \log(Bt+C) 
        }_{t=X}^{t=Y} + O\lrp{\Delta^{1/2+\varepsilon}}.
    \end{align}

    In the following, we will only consider $n \ge 1,~ \gcd(m,n)=1$ (and omit these assumptions from our set notation). Note that $B mn + C n^2 > 0$ always holds since $\frac mn \ge X > \frac{-C}{B}$. And observe that $B mn + C n^2 \le \Delta$ iff $\frac mn \le \frac{\Delta}{Bn^2}-\frac{C}{B}$. This means that
    \begin{align} 
        \#\, W_{\Delta}^{(0,B,C)} \cap [X,Y] 
        &= \sum_{n \ge 1} \#\lrcb{ X \le \frac{m}{n} \le Y ~:~  0 < B mn + C n^2 \le \Delta} \\
        &= \sum_{n \ge 1} \#\lrcb{ X \le \frac{m}{n} \le Y ~:~ \frac{m}{n} \le  \frac{\Delta}{Bn^2}-\frac{C}{B}}.
    \end{align}
    Then note that 
    \begin{align}
        Y &\le  \frac{\Delta}{Bn^2}-\frac{C}{B} \quad\text{iff}\quad n \le \sqrt{\frac{\Delta}{BY+C}}, \\
        X &\le  \frac{\Delta}{Bn^2}-\frac{C}{B} \quad\text{iff}\quad n \le \sqrt{\frac{\Delta}{BX+C}},
    \end{align}
    which means that
    \begin{align} 
        &\#\, W_{\Delta}^{(0,B,C)} \cap [X,Y]  \\
        &=
        \sum_{n \le \sqrt{\frac{\Delta}{BY+C}}} \#\lrcb{ X \le \frac{m}{n} \le Y }  \\
        &\quad+
        \sum_{\sqrt{\frac{\Delta}{BY+C}} < n \le \sqrt{\frac{\Delta}{BX+C}}} \#\lrcb{ X \le \frac{m}{n} \le   \frac{\Delta}{Bn^2}-\frac{C}{B} } \\
        &= \sum_{n \le \sqrt{\frac{\Delta}{BY+C}}} \lrb{\phi(n) \Big(Y-X\Big) + O\lrp{n^{\varepsilon}}} 
        \quad \text{uniformly in $n$ \qquad (Lemma \ref{lem:count-coprime-range})} \\
        &\quad+
        \sum_{\sqrt{\frac{\Delta}{BY+C}} < n \le \sqrt{\frac{\Delta}{BX+C}}}  \lrb{\phi(n) \lrp{
            \frac{\Delta}{Bn^2} - \frac{C}{B} - X} + O\lrp{n^{\varepsilon}} 
        } \\
        &=  \sum_{n \le \sqrt{\frac{\Delta}{BY+C}}} \phi(n) Y
        ~-~
        \sum_{n \le \sqrt{\frac{\Delta}{BX+C}}} \phi(n) X
        ~-~
        \frac{C}{B} \sum_{\sqrt{\frac{\Delta}{BY+C}} < n \le \sqrt{\frac{\Delta}{BX+C}}} \phi(n)
        \\
        &\quad+\quad
        \frac{\Delta}{B} \sum_{\sqrt{\frac{\Delta}{BY+C}} < n \le \sqrt{\frac{\Delta}{BX+C}}}  \frac{\phi(n)}{n^2}  \quad+\quad O\lrp{\Delta^{1/2+\varepsilon}} \\
        &=\, \frac{3}{\pi^2} \lrb{t\frac{\Delta}{Bt+C}}_{t=X}^{t=Y} 
        ~-~ \frac{C}{B} \frac{3}{\pi^2}
        \lrb{\frac{\Delta}{Bt+C}}_{t=Y}^{t=X}
        \qquad\qquad\quad\ \ \, \text{(Lemma \ref{lem:sum-phi})}
        \\
        &\quad+
        \frac{\Delta}{B} \frac{6}{\pi^2}\lrb{ 
            \log \sqrt{\frac{\Delta}{Bt+C}} ~+~ \gamma_0
        }_{t=Y}^{t=X} + O\lrp{\Delta^{1/2+\varepsilon}} \qquad\quad \text{(Lemma \ref{lem:sum-phi-over-a2})} \\
        &=\, \frac{3\Delta}{\pi^2} \lrb{
            \frac{t+\frac{C}{B}}{Bt+C}
        }_{t=X}^{t=Y}
        ~+~
        \frac{6\Delta}{B\pi^2} \lrb{ 
         \log \sqrt{\frac{1}{Bt+C}} 
        }_{t=Y}^{t=X} + O\lrp{\Delta^{1/2+\varepsilon}} \\
        &= \frac{3\Delta}{\pi^2} \lrb{ 
        \frac1B \log (Bt+C) 
        }_{t=X}^{t=Y} + O\lrp{\Delta^{1/2+\varepsilon}},
    \end{align}
    verifying \eqref{eqn:goal-equid-A=0,B>0}.
\end{proof}

{
\renewcommand{\thetheorem}{\ref{thm:equid-WABC} (\(A=0,B<0\))}
\begin{theorem} \label{thm:equid-W--A=0,B<0}
    Fix $(B,C) \in \RR^2$ with $B < 0$. Then
    \begin{align}
        W_{\Delta}^{(0,B,C)} = \lrcb{\frac{m}{n}  ~:~ n\ge1,~ \gcd(m,n)=1,~ 0 < B \lrp{\frac mn} + C \le \frac{\Delta}{n^2}} 
    \end{align} 
    is equidistributed over $(-\infty,\frac{-C}{B})$ with respect to the measure $d\mu = \frac{1}{Bt+C} \,dt$.
\end{theorem}
}
\begin{proof}
    Observe that 
    \begin{align}
        -W_{\Delta}^{(0,B,C)} 
        &= \lrcb{\frac{-m}{n}  ~:~ n\ge1,~ \gcd(m,n)=1,~ 0 < B mn + C n^2 \le \Delta} \\
        &= \lrcb{\frac{m}{n}  ~:~ n\ge1,~ \gcd(m,n)=1,~ 0 < B (-m)n + C n^2 \le \Delta} \\
        &= W_{\Delta}^{(0,-B,C)}.
    \end{align} 
    Hence by Theorem \ref{thm:equid-WABC} (Case $A = 0, B > 0$), $-W_{\Delta}^{(0,B,C)} = W_{\Delta}^{(0,-B,C)}$ is equidistributed over $(\frac CB, \infty)$ with respect to the measure $d\mu = \frac{1}{-Bt+C}dt$. Applying the transformation $t \mapsto -t$ then yields the desired result.
\end{proof}

%%%%%%%%%%%%%%%%%%%%%%%%%%%%%%%%%%%%%%%%%%%
%%%%%%%%%%%%%%%%%%%%%%%%%%%%%%%%%%%%%%%%%%%

\section{Proof of Theorem \ref{thm:equid-WABC} for \texorpdfstring{$A > 0$}{A>0}}
\label{sec:equid-WABC--A>0}

%%%%%%%%%%%%%%%%%%%%%%%%%%%%%%%%%%%

{
\renewcommand{\thetheorem}{\ref{thm:equid-WABC} (\(A>0,D>0\))}
\begin{theorem} \label{thm:equid-W--A>0,D>0}
Fix $(A,B,C) \in \RR^3$ with $A > 0$ and $D := B^2-4AC > 0$. Then
\begin{align}
    W_{\Delta}^{(A,B,C)} := \Big\{\frac{m}{n} ~:~ n\ge 1,~ \gcd(m,n)=1,~   0 < A m^2 + B mn + C n^2  \le \Delta
    \Big\}
\end{align}
is equidistributed over $(\frac{-B+\sqrt{D}}{2A},\infty] \cup [-\infty,\frac{-B-\sqrt{D}}{2A})$ with respect to the measure $d\mu = \frac{1}{At^2+Bt+C} \,dt$. Here, the interval $(\frac{-B+\sqrt{D}}{2A},\infty] \cup [-\infty,\frac{-B-\sqrt{D}}{2A})$ is understood to wrap around from $\infty$ to $-\infty$.
\end{theorem}
}
\begin{proof}
To prove equidistribution, we show that for all $[X,Y] \subseteq (\frac{-B+\sqrt{D}}{2A},\infty] \cup [-\infty,\frac{-B-\sqrt{D}}{2A})$ (including those subintervals wrapping around from $\infty$ to $-\infty$),
\begin{align} \label{eqn:goal-equid-A>0,D>0} \tag{$**$}
    \#\, W_{\Delta}^{(A,B,C)} \cap [X,Y] =
    \frac{3\Delta}{\pi^2}
    \Bigg[ 
        \frac{1}{\sqrt D}
        \log\lrp{
        \frac{
            t - \frac{-B+\sqrt{D}}{2A}
        }{
            t - \frac{-B-\sqrt{D}}{2A}
        }
        }
    \Bigg]_{t=X}^{t=Y} 
         ~+~ O(\Delta^{1/2+\varepsilon}).
\end{align}

%%%%%%%%%%%%%%%%%%%%%%%%%%%%%%%%%%

It suffices to show \eqref{eqn:goal-equid-A>0,D>0} for $[X,Y] \subseteq (\frac{-B+\sqrt{D}}{2A},\infty]$ and for $[X,Y] \subseteq [-\infty, \frac{-B-\sqrt{D}}{2A})$.

\noindent
\textbf{Part 1: $[X,Y] \subseteq (\frac{-B+\sqrt{D}}{2A},\infty]$.}

In the following, we will only consider $n \ge 1,~ \gcd(m,n)=1$ (and omit these assumptions from our set notation).
For this part,  
\begin{align}
    W_{\Delta}^{(A,B,C)} \cap [X,Y] 
    &= \lrcb{X \le \frac{m}{n} \le Y ~:~  0 < A m^2 + B mn + C n^2  \le \Delta
    } \\
    &= \lrcb{X \le \frac{m}{n} \le Y ~:~  0 < A \lrp{\frac{m}{n}}^2 + B \lrp{\frac{m}{n}} + C  \le \frac{\Delta}{n^2}
    } \\
    &= \lrcb{X \le \frac{m}{n} \le Y ~:~  \frac{m}{n} \le \frac{-B+\sqrt{D+4A\frac{\Delta}{n^2}}}{2A}
    }. 
\end{align}
Then note that
\begin{align}
    &Y \le  \frac{-B+\sqrt{D+4A\frac{\Delta}{n^2}}}{2A}
     \qquad \text{iff}\qquad  n \le \sqrt{\frac{\Delta}{AY^2+BY+C}},
     \\
    &X \le  \frac{-B+\sqrt{D+4A\frac{\Delta}{n^2}}}{2A}
     \qquad \text{iff}\qquad  n \le \sqrt{\frac{\Delta}{AX^2+BX+C}},
\end{align}
which means that
\begin{align}
    & \#\, W_{\Delta}^{(A,B,C)} \cap [X,Y] \\
    &=
    \sum_{ n \le \sqrt{\frac{\Delta}{AY^2+BY+C}} }
    \# \lrcb{X \le \frac{m}{n} \le Y
    }
    \\
    &\quad+
    \sum_{
    \sqrt{\frac{\Delta}{AY^2+BY+C}} < n \le \sqrt{\frac{\Delta}{AX^2+BX+C}}}
    \#\lrcb{X \le \frac{m}{n} \le \frac{-B+\sqrt{D+4A\frac{\Delta}{n^2}}}{2A}   
    } \\
    &=
    \sum_{ n \le \sqrt{\frac{\Delta}{AY^2+BY+C}} } 
    \lrb{
        \phi(n) \Big(Y-X\Big) + O(n^{\varepsilon})
    } \qquad \text{uniformly in $n$ \qquad\qquad (Lemma \ref{lem:count-coprime-range})} \\
    &\quad+
    \sum_{ \sqrt{\frac{\Delta}{AY^2+BY+C}} < n \le \sqrt{\frac{\Delta}{AX^2+BX+C}} }
    \lrb{\phi(n) \lrp{ 
        \frac{-B+\sqrt{D+4A\frac{\Delta}{n^2}}}{2A}
        - X
    }
    + O(n^{\varepsilon})
    }  \\
    &=
    \sum_{ n \le \sqrt{\frac{\Delta}{AY^2+BY+C}}} 
    \phi(n)Y \ \ - \sum_{ n \le \sqrt{\frac{\Delta}{AX^2+BX+C}}} 
    \phi(n)X \quad  \\
    &\quad+
    \sum_{ \sqrt{\frac{\Delta}{AY^2+BY+C}} < n \le \sqrt{\frac{\Delta}{AX^2+BX+C}} }
    \phi(n)  
    \frac{-B+\sqrt{D+4A\frac{\Delta}{n^2}}}{2A}
    \quad+ \quad O(\Delta^{1/2+\varepsilon})
    \label{eqn:temp-equid-sum-sqrt-A>0,D>0}
    \\
    &= \frac{3}{\pi^2} \lrb{ t  \frac{\Delta}{At^2+Bt+C}}_{t=X}^{t=Y} ~+~ \frac{-B}{2A}\frac{3}{\pi^2} \lrb{  \frac{\Delta}{At^2+Bt+C}}_{t=Y}^{t=X} \qquad\quad \text{(Lemmas \ref{lem:sum-phi} and \ref{lem:sum-phi-sqrt})}
    \\
    &\quad +~ \frac{1}{2A} \frac{3\Delta}{\pi^2}\Bigg[ \sqrt{Du^2+4Au} +  \frac{4A}{\sqrt D}
            \log\lrp{
                \sqrt{Du} + \sqrt{Du+4A}
            }
         \Bigg]_{u=\frac{1}{AY^2+BY+C}}^{u=\frac{1}{AX^2+BX+C}} ~+~ O(\Delta^{1/2+\varepsilon}) \\
    &= \frac{3\Delta}{\pi^2} \lrb{ \frac{{t+\frac{B}{2A}}}{At^2+Bt+C}}_{t=X}^{t=Y} 
    ~+~ \frac{3\Delta}{\pi^2} \lrb{
        \frac{1}{2A} \sqrt{Du^2+4Au}
    }_{u=\frac{1}{AY^2+BY+C}}^{u=\frac{1}{AX^2+BX+C}}
    \\
    &\quad +~ \frac{3\Delta}{\pi^2}\Bigg[ 
            \frac{2}{\sqrt D}
            \log\lrp{
                \sqrt{Du} + \sqrt{Du+4A}
            }
         \Bigg]_{u=\frac{1}{AY^2+BY+C}}^{u=\frac{1}{AX^2+BX+C}} ~+~ O(\Delta^{1/2+\varepsilon}). 
\end{align}

Observe here that for $u = \frac{1}{At^2+Bt+C}$ with $t \ge \frac{-B}{2A}$, 
\begin{align}
    \sqrt{Du + 4A} &= \sqrt{\frac{(B^2-4AC) + 4A(At^2+Bt+C)}{At^2+Bt+C}} = \sqrt{\frac{B^2 + 4A^2t^2 + 4ABt}{At^2+Bt+C}} = \frac{2At+B}{\sqrt{At^2+Bt+C}},
\end{align}
so that
\begin{align}
    \lrb{
        \frac{1}{2A} \sqrt{Du^2+4Au}
    }_{u=\frac{1}{AY^2+BY+C}}^{u=\frac{1}{AX^2+BX+C}} = \lrb{\frac{t+\frac{B}{2A}}{At^2+Bt+C}}_{t=Y}^{t=X}.
\end{align}
This then means that 
\begin{align}
     & \#\, W_{\Delta}^{(A,B,C)} \cap [X,Y] \\ 
    &=  \frac{3\Delta}{\pi^2}\Bigg[ 
            \frac{2}{\sqrt D}
            \log\lrp{
                \sqrt{Du} + \sqrt{Du+4A}
            }
         \Bigg]_{u=\frac{1}{AY^2+BY+C}}^{u=\frac{1}{AX^2+BX+C}} ~+~ O(\Delta^{1/2+\varepsilon}) \label{eqn:tempVABC-log-formula} \\
    &=  \frac{3\Delta}{\pi^2}\Bigg[
            \frac{2}{\sqrt D}
            \log\lrp{  
                \frac{\sqrt D + 2At+B}{\sqrt{At^2+Bt+C}}
            }
         \Bigg]_{t=Y}^{t=X} 
        ~+~ O(\Delta^{1/2+\varepsilon}) \\
    &=  \frac{3\Delta}{\pi^2}\Bigg[ 
            \frac{2}{\sqrt D} 
            \log\lrp{
                \frac{
                    2A\lrp{
                        t - \frac{-B-\sqrt{D}}{2A}
                    }
                }{
                    \sqrt{At^2+Bt+C}
                }
            }
         \Bigg]_{t=Y}^{t=X} 
         ~+~ O(\Delta^{1/2+\varepsilon}) \\
    &=  \frac{3\Delta}{\pi^2}\Bigg[ 
            \frac{1}{\sqrt D} 
            \log\lrp{
                \frac{
                    4A^2\lrp{
                        t - \frac{-B-\sqrt{D}}{2A}
                    }^2
                }{
                    At^2+Bt+C
                }
            }
         \Bigg]_{t=Y}^{t=X} 
         ~+~ O(\Delta^{1/2+\varepsilon}) \\
    &=  \frac{3\Delta}{\pi^2}\Bigg[ 
            \frac{1}{\sqrt D} \log\lrp{\frac{{
                t - \frac{-B-\sqrt{D}}{2A}
            }}{t - \frac{-B+\sqrt{D}}{2A}}}
         \Bigg]_{t=Y}^{t=X} 
         ~+~ O(\Delta^{1/2+\varepsilon}) \\
    &= \frac{3\Delta}{\pi^2}
    \Bigg[ 
        \frac{1}{\sqrt D}
        \log\lrp{
        \frac{
            t - \frac{-B+\sqrt{D}}{2A}
        }{
            t - \frac{-B-\sqrt{D}}{2A}
        }
        }
    \Bigg]_{t=X}^{t=Y} 
         ~+~ O(\Delta^{1/2+\varepsilon}),
\end{align}
verifying \eqref{eqn:goal-equid-A>0,D>0}.

%%%%%%%%%%%%%%%%%%%%%%%%%%%%%%%

\noindent
\textbf{Part 2: $[X,Y] \subseteq [-\infty, \frac{-B-\sqrt{D}}{2A})$. }

For this part, we use symmetry to reduce to Part 1 (for parameters $A,-B,C$). In particular, we have 
\begin{align}
    \#\, W_{\Delta}^{(A,B,C)} \cap [X,Y] 
    &= \#\lrcb{X \le \frac{m}{n} \le Y ~:~   0 < A m^2 + B mn + C n^2  \le \Delta
    } \\
    &= \#\lrcb{X \le \frac{-m}{n} \le Y ~:~  0 < A (-m)^2 + B (-m)n + Cn^2  \le \Delta
    } \\
    &= \#\lrcb{-Y \le \frac{m}{n} \le -X ~:~  0 < A m^2 - B mn + C n^2  \le \Delta
    } \\
    &= \#\, W_{\Delta}^{(A,-B,C)} \cap [-Y,-X] \\
    &= \frac{3\Delta}{\pi^2}
    \Bigg[ 
        \frac{1}{\sqrt D}
        \log\lrp{
        \frac{
            t - \frac{B+\sqrt{D}}{2A}
        }{
            t - \frac{B-\sqrt{D}}{2A}
        }
        }
    \Bigg]_{t=-Y}^{t=-X} 
         ~+~ O(\Delta^{1/2+\varepsilon}) \qquad \text{(by Part 1)} \\
    &= \frac{3\Delta}{\pi^2}
    \Bigg[ 
        \frac{1}{\sqrt D}
        \log\lrp{
        \frac{
            t - \frac{-B-\sqrt{D}}{2A}
        }{
            t - \frac{-B+\sqrt{D}}{2A}
        }
        }
    \Bigg]_{t=Y}^{t=X} 
         ~+~ O(\Delta^{1/2+\varepsilon})  \\
    &= \frac{3\Delta}{\pi^2}
    \Bigg[ 
        \frac{1}{\sqrt D}
        \log\lrp{
        \frac{
            t - \frac{-B+\sqrt{D}}{2A}
        }{
            t - \frac{-B-\sqrt{D}}{2A}
        }
        }
    \Bigg]_{t=X}^{t=Y} 
         ~+~ O(\Delta^{1/2+\varepsilon}),
\end{align}
verifying \eqref{eqn:goal-equid-A>0,D>0}.
\end{proof}

{
\renewcommand{\thetheorem}{\ref{thm:equid-WABC} (\(A>0,D<0\))}
\begin{theorem}  \label{thm:equid-W--A>0,D<0}
Fix $(A,B,C) \in \RR^3$ with $A > 0$ and $D := B^2-4AC < 0$. Then
\begin{align}
    W_{\Delta}^{(A,B,C)} := \Big\{\frac{m}{n} ~:~ n\ge 1,~ \gcd(m,n)=1,~   0 < A m^2 + B mn + C n^2  \le \Delta
    \Big\}
\end{align}
is equidistributed over  $[-\infty,\infty]$ 
with respect to the measure $d\mu = \frac{1}{At^2+Bt+C} \,dt$.
\end{theorem}
}
\begin{proof}
To prove equidistribution, we show that for all $[X,Y] \subseteq [-\infty,\infty]$,
\begin{align} \label{eqn:goal-equid-A>0,D<0} \tag{$*{*}*$}
    \#\, W_{\Delta}^{(A,B,C)} \cap [X,Y] = 
    \frac{3\Delta}{\pi^2}\Bigg[ 
        \frac{2}{\sqrt{-D}} \arctan\!\lrp{\frac{2At+B}{\sqrt{-D}}}
         \Bigg]_{t=X}^{t=Y} ~+~ O(\Delta^{1/2+\varepsilon}).
\end{align}

%%%%%%%%%%%%%%%%%%

It suffices to show \eqref{eqn:goal-equid-A>0,D<0} for 
$[X,Y] \subseteq [\frac{-B}{2A},\infty]$ and for $[X,Y] \subseteq [-\infty,\frac{-B}{2A}]$.

\noindent
\textbf{Part 1: $[X,Y] \subseteq [\frac{-B}{2A}, \infty]$.}

Note that the result of Lemma \ref{lem:sum-phi-sqrt-neg-discr} is the same as Lemma \ref{lem:sum-phi-sqrt}, except with the function
\begin{align}
    \frac{-1}{\sqrt{-D}} \arctan\!\lrp{ \frac{\sqrt{Du+4A}}{\sqrt{-Du}}}
    \qquad \text{replacing} \qquad
    \frac{1}{\sqrt{D}} \log\lrp{
                \sqrt{Du} + \sqrt{Du+4A}
            }.
\end{align}

By an identical argument as for \eqref{eqn:tempVABC-log-formula}, we have that
\begin{align}
    \#\, W_{\Delta}^{(A,B,C)} \cap [X,Y] 
    &=\frac{3\Delta}{\pi^2}\Bigg[ 
            \frac{-2}{\sqrt{-D}} \arctan\lrp{\frac{\sqrt{Du+4A}}{\sqrt{-Du}}}
         \Bigg]_{u=\frac{1}{AY^2+BY+C}}^{u=\frac{1}{AX^2+BX+C}} ~+~ O(\Delta^{1/2+\varepsilon}) \\
    &=\frac{3\Delta}{\pi^2}\Bigg[ 
        \frac{2}{\sqrt{-D}} \arctan\lrp{\frac{2At+B}{\sqrt{-D}}}
         \Bigg]_{t=X}^{t=Y} ~+~ O(\Delta^{1/2+\varepsilon}),
\end{align}
verifying \eqref{eqn:goal-equid-A>0,D<0}.

%%%

\noindent
\textbf{Part 2: $[X,Y] \subseteq [-\infty, \frac{-B}{2A}]$. }

For this part, one can use an identical symmetry argument as in Theorem \ref{thm:equid-W--A>0,D>0} to reduce to Part 1 (for parameters $A,-B,C$).
\end{proof}

%%%%%%%%%%%%%%%%%%%%%%%%%%%%%%

{
\renewcommand{\thetheorem}{\ref{thm:equid-WABC} (\(A>0,D=0\))}
\begin{theorem} \label{thm:equid-W--A>0,D=0}
    Fix $(A,B,C) \in \RR^3$ with $A > 0$ and $D := B^2-4AC=0$. Then
    \begin{align}
        W_{\Delta}^{(A,B,C)} = \lrcb{\frac{m}{n}  ~:~ n\ge1,~ \gcd(m,n)=1,~ 0 < A m^2 +B mn + C n^2 \le \Delta} 
    \end{align} 
    is equidistributed over $(\frac{-B}{2A},\infty]\cup[-\infty,\frac{-B}{2A})$ with respect to the measure $d\mu = \frac{1}{At^2+Bt+C} \,dt$. Here, the interval $(\frac{-B}{2A},\infty]\cup[-\infty,\frac{-B}{2A})$ is understood to wrap around from $\infty$ to $-\infty$.
\end{theorem}
\addtocounter{theorem}{-1}
}
\begin{proof}
% Fix $[K_1,\infty) \cup (-\infty, K_2] \subseteq (\frac{-B}{2A},\infty) \cup (-\infty,\frac{-B}{2A})$ and $\varepsilon > 0$.

To prove equidistribution, we show that for all $[X,Y] \subseteq (\frac{-B}{2A},\infty] \cup [-\infty,\frac{-B}{2A})$ (including those subintervals wrapping around from $\infty$ to $-\infty$),
\begin{align} \label{eqn:goal-equid-A>0,D=0} \tag{$*{*}{*}*$}
    \#\, W_{\Delta}^{(A,B,C)} \cap [X,Y] =
    \frac{3\Delta}{\pi^2}
    \Bigg[ 
        \frac{-2}{2At+B}
    \Bigg]_{t=X}^{t=Y} 
         ~+~ O(\Delta^{1/2+\varepsilon}).
\end{align}

%%%%%%%%%%%%%%%%%%%%%%%%%%%%%%%%%%

It suffices to show \eqref{eqn:goal-equid-A>0,D=0} for $[X,Y] \subseteq (\frac{-B}{2A},\infty]$ and for $[X,Y] \subseteq [-\infty,\frac{-B}{2A})$.

\noindent
\textbf{Part 1: $[X,Y] \subseteq (\frac{-B}{2A},\infty]$.}

By an identical argument as for \eqref{eqn:temp-equid-sum-sqrt-A>0,D>0}, we have that
\begin{align}
    & \#\, W_{\Delta}^{(A,B,C)} \cap [X,Y] \\
    &=
    \sum_{ n \le \sqrt{\frac{\Delta}{AY^2+BY+C}}} 
    \phi(n)Y \ \ - \sum_{ n \le \sqrt{\frac{\Delta}{AX^2+BX+C}}} 
    \phi(n)X \quad  \\
    &\quad+
    \sum_{ \sqrt{\frac{\Delta}{AY^2+BY+C}} < n \le \sqrt{\frac{\Delta}{AX^2+BX+C}} }
    \lrp{
        \frac{-B}{2A} \phi(n) 
        +\frac{\sqrt \Delta}{\sqrt A} \frac{\phi(n)}{n}
    }
    \quad+ \quad O(\Delta^{1/2+\varepsilon}) \\
    &= \frac{3}{\pi^2} \lrb{ t  \frac{\Delta}{At^2+Bt+C}}_{t=X}^{t=Y} ~+~ \frac{-B}{2A}\frac{3}{\pi^2} \lrb{  \frac{\Delta}{At^2+Bt+C}}_{t=Y}^{t=X} \qquad\quad \text{(Lemma \ref{lem:sum-phi})}
    \\
    &\quad +~ \frac{\sqrt \Delta}{\sqrt A} \frac{3}{\pi^2}
        \Bigg[ 
        \sqrt{
            \frac{\Delta}{At^2+Bt+C}    
        }
        \Bigg]_{t=Y}^{t=X} 
    ~+~ O(\Delta^{1/2+\varepsilon})
    \qquad\qquad\qquad\quad\text{(Lemma \ref{lem:sum-phi-over-a})}
    \\
    &= \frac{3\Delta}{\pi^2} \lrb{   \frac{t- \frac{-B}{2A}}{At^2+Bt+C}}_{t=X}^{t=Y} ~+~ 
    \frac{6 \Delta}{\pi^2} 
        \Bigg[ 
            \frac{1}{\sqrt A}
            \frac{1}{\sqrt{At^2+Bt+C}}  
        \Bigg]_{t=Y}^{t=X} ~+~ O(\Delta^{1/2+\varepsilon})
        \\
    &= \frac{3\Delta}{\pi^2} \lrb{    \frac{1}{A\lrp{t - \frac{-B}{2A}}}}_{t=X}^{t=Y} ~+~ 
    \frac{6 \Delta}{\pi^2} 
        \Bigg[ 
            \frac{1}{A\lrp{t- \frac{-B}{2A}}}   
        \Bigg]_{t=Y}^{t=X} ~+~ O(\Delta^{1/2+\varepsilon}) \quad \text{(since $D=0$)}
        \\
    &= \frac{3\Delta}{\pi^2} \lrb{    \frac{-2}{2At+B}}_{t=X}^{t=Y} ~+~ O(\Delta^{1/2+\varepsilon}),
\end{align}
verifying \eqref{eqn:goal-equid-A>0,D=0}.

%%%

\noindent
\textbf{Part 2: $[X,Y] \subseteq [-\infty,\frac{-B}{2A})$. }

For this part, one can use an identical symmetry argument as in Theorem \ref{thm:equid-W--A>0,D>0} to reduce to Part 1 (for parameters $A,-B,C$).
\end{proof}

%%%%%%%%%%%%%%%%%%%%%%%%%%%%%

\section{Proof of Theorem \ref{thm:equid-WABC} for \texorpdfstring{$A<0$}{A<0}}
\label{sec:equid-WABC--A<0}

{
\renewcommand{\thetheorem}{\ref{thm:equid-WABC} (\(A<0,D>0\))}
\begin{theorem} \label{thm:equid-W--A<0,D>0}
Fix $(A,B,C) \in \RR^3$ with $A < 0$ and $D := B^2-4AC > 0$. Then
\begin{align}
    W_{\Delta}^{(A,B,C)} := \Big\{\frac{m}{n} ~:~ n\ge 1,~ \gcd(m,n)=1,~   0 < A m^2 + B mn + C n^2  \le \Delta
    \Big\}
\end{align}
is equidistributed over  $(\frac{-B+\sqrt{D}}{2A},~ \frac{-B-\sqrt{D}}{2A})$ with respect to the measure $d\mu = \frac{1}{At^2+Bt+C} \,dt$.
\end{theorem}
}
\begin{proof}
% Fix $[K_1, K_2] \subseteq (\frac{-B+\sqrt{D}}{2A},~ \frac{-B-\sqrt{D}}{2A})$ and $\varepsilon > 0$. 
To prove equidistribution, we show that for all $[X, Y] \subseteq (\frac{-B+\sqrt{D}}{2A},~ \frac{-B-\sqrt{D}}{2A})$,
\begin{align} \label{eqn:goal-equid-A<0,D>0} \tag{$*{*}{*}{*}*$}
    \#\, W_{\Delta}^{(A,B,C)} \cap [X,Y] = 
    \frac{3\Delta}{\pi^2}
    \Bigg[ 
        \frac1{\sqrt D}
        \log\lrp{
        -\frac{
            t - \frac{-B+\sqrt{D}}{2A}
        }{
            t - \frac{-B-\sqrt{D}}{2A}
        }
        }
    \Bigg]_{t=X}^{t=Y} 
         ~+~ O(\Delta^{1/2+\varepsilon}).
\end{align}

%%%%%%%%%%%%%%%%%%%%%%%%%%%%%%%%%%%

It suffices to show \eqref{eqn:goal-equid-A<0,D>0} for $[X,Y] \subseteq [\frac{-B}{2A}, \frac{-B-\sqrt{D}}{2A})$ and for $[X,Y] \subseteq (\frac{-B+\sqrt{D}}{2A}, \frac{-B}{2A}]$.

\noindent
\textbf{Part 1: $[X,Y] \subseteq [\frac{-B}{2A}, \frac{-B-\sqrt{D}}{2A})$.}

In the following, we will only consider $n \ge 1,~ \gcd(m,n)=1$, and omit these assumptions from our set notation.
For this part, we have
\begin{align}
    W_{\Delta}^{(A,B,C)} \cap [X,Y] 
    &= \lrcb{X \le \frac{m}{n} \le Y ~:~  0 < A \lrp{\frac{m}{n}}^2 + B \lrp{\frac{m}{n}} + C  \le \frac{\Delta}{n^2}
    } \\
    &= \lrcb{X \le \frac{m}{n} \le Y ~:~  \frac{m}{n} \ge \frac{-B-\sqrt{\max\big(0,D+4A\frac{\Delta}{n^2}\big)}}{2A}
    }. 
\end{align}
Then note that
\begin{align}
    &X \ge  \frac{-B-\sqrt{\max\big(0,D+4A\frac{\Delta}{n^2}\big)}}{2A}
     \qquad \text{iff}\qquad  n \le \sqrt{\frac{\Delta}{AX^2+BX+C}},
     \\
    &Y \ge  \frac{-B-\sqrt{\max\big(0,D+4A\frac{\Delta}{n^2}\big)}}{2A}
     \qquad \text{iff}\qquad  n \le \sqrt{\frac{\Delta}{AY^2+BY+C}},
\end{align}
which means that
\begin{align}
    & \#\, W_{\Delta}^{(A,B,C)} \cap [X,Y] \\
    &=
    \sum_{ n \le \sqrt{\frac{\Delta}{AX^2+BX+C}} }
    \# \lrcb{X \le \frac{m}{n} \le Y
    }
    \\
    &\quad+
    \sum_{
    \sqrt{\frac{\Delta}{AX^2+BX+C}} < n \le \sqrt{\frac{\Delta}{AY^2+BY+C}}}
    \#\lrcb{\frac{-B-\sqrt{D+4A\frac{\Delta}{n^2}}}{2A}  
    \le \frac{m}{n} \le Y
    } \\
    &=
    \sum_{ n \le \sqrt{\frac{\Delta}{AX^2+BX+C}} } 
    \lrb{
        \phi(n) \Big(Y-X\Big) + O(n^{\varepsilon})
    } \qquad \text{uniformly in $n$ \qquad\qquad (Lemma \ref{lem:count-coprime-range})} \\
    &\quad+
    \sum_{ \sqrt{\frac{\Delta}{AX^2+BX+C}} < n \le \sqrt{\frac{\Delta}{AY^2+BY+C}} }
    \lrb{\phi(n) \lrp{ 
        Y - 
        \frac{-B-\sqrt{D+4A\frac{\Delta}{n^2}}}{2A}
    }
    + O(n^{\varepsilon})
    }  \\
    &=
    \sum_{ n \le \sqrt{\frac{\Delta}{AY^2+BY+C}}} 
    \phi(n)Y \ \ - \sum_{ n \le \sqrt{\frac{\Delta}{AX^2+BX+C}}} 
    \phi(n)X \quad  \\
    &\quad-
    \sum_{ \sqrt{\frac{\Delta}{AX^2+BX+C}} < n \le \sqrt{\frac{\Delta}{AY^2+BY+C}} }
    \phi(n)  
    \frac{-B-\sqrt{D+4A\frac{\Delta}{n^2}}}{2A}
    \quad+ \quad O(\Delta^{1/2+\varepsilon})\\
    &= \frac{3}{\pi^2} \lrb{ t  \frac{\Delta}{At^2+Bt+C}}_{t=X}^{t=Y} ~-~ \frac{-B}{2A}\frac{3}{\pi^2} \lrb{  \frac{\Delta}{At^2+Bt+C}}_{t=X}^{t=Y} \qquad\quad \text{(Lemmas \ref{lem:sum-phi} and \ref{lem:sum-phi-sqrt})}
    \\
    &\quad +~ \frac{1}{2A} \frac{3\Delta}{\pi^2}\Bigg[ \sqrt{Du^2+4Au} +  \frac{4A}{\sqrt D}
            \log\lrp{
                \sqrt{Du} + \sqrt{Du+4A}
            }
         \Bigg]_{u=\frac{1}{AX^2+BX+C}}^{u=\frac{1}{AY^2+BY+C}} ~+~ O(\Delta^{1/2+\varepsilon}) \\
    &= \frac{3\Delta}{\pi^2} \lrb{ \frac{t-\frac{-B}{2A}}{At^2+Bt+C}}_{t=X}^{t=Y} 
    ~+~ \frac{3\Delta}{\pi^2} \lrb{
        \frac{1}{2A} \sqrt{Du^2+4Au}
    }_{u=\frac{1}{AX^2+BX+C}}^{u=\frac{1}{AY^2+BY+C}}
    \\
    &\quad +~ \frac{3\Delta}{\pi^2}\Bigg[
            \frac{2}{\sqrt D}
            \log\lrp{
                \sqrt{Du} + \sqrt{Du+4A}
            }
         \Bigg]_{u=\frac{1}{AX^2+BX+C}}^{u=\frac{1}{AY^2+BY+C}} ~+~ O(\Delta^{1/2+\varepsilon}). 
\end{align}

Observe here that for $u = \frac{1}{At^2+Bt+C}$ with $t \ge \frac{-B}{2A}$, 
\begin{align}
    \sqrt{Du +4A} &= \sqrt{\frac{(B^2-4AC) + 4A(At^2+Bt+C)}{At^2+Bt+C}} = \sqrt{\frac{B^2 + 4A^2t^2 + 4ABt}{At^2+Bt+C}} = \frac{-2At-B}{\sqrt{At^2+Bt+C}},
\end{align}
so that
\begin{align}
    \lrb{
        \frac{1}{2A} \sqrt{Du^2+4Au}
    }_{u=\frac{1}{AX^2+BX+C}}^{u=\frac{1}{AY^2+BY+C}} = \lrb{\frac{-(t-\frac{-B}{2A})}{At^2+Bt+C}}_{t=X}^{t=Y}.
\end{align}
This then means that
\begin{align}
    & \#\, W_{\Delta}^{(A,B,C)} \cap [X,Y] \\
    &=  \frac{3\Delta}{\pi^2}\Bigg[ 
            \frac{2}{\sqrt D}
            \log\lrp{
                \sqrt{Du} + \sqrt{Du+4A}
            }
         \Bigg]_{u=\frac{1}{AX^2+BX+C}}^{u=\frac{1}{AY^2+BY+C}} ~+~ O(\Delta^{1/2+\varepsilon}) \\
    &=  \frac{3\Delta}{\pi^2}\Bigg[ 
            \frac{2}{\sqrt D}
            \log\lrp{\frac{\sqrt{D} - 2At-B}{\sqrt{At^2+Bt+C}}}  
         \Bigg]_{t=X}^{t=Y} 
        ~+~ O(\Delta^{1/2+\varepsilon}) \\
    &=  \frac{3\Delta}{\pi^2}\Bigg[ 
            \frac{2}{\sqrt D} 
            \log\lrp{\frac{-2A\lrp{
                t - \frac{-B+\sqrt{D}}{2A}
            }}{\sqrt{At^2+Bt+C}}}
         \Bigg]_{t=X}^{t=Y} 
         ~+~ O(\Delta^{1/2+\varepsilon}) \\
    &=  \frac{3\Delta}{\pi^2}\Bigg[ 
            \frac{1}{\sqrt D} 
            \log\lrp{\frac{4A^2\lrp{
                t - \frac{-B+\sqrt{D}}{2A}
            }^2}{At^2+Bt+C}}
         \Bigg]_{t=X}^{t=Y} 
         ~+~ O(\Delta^{1/2+\varepsilon}) \\
    &=  \frac{3\Delta}{\pi^2}
    \Bigg[ 
        \frac{1}{\sqrt D}
        \log\lrp{
        -\frac{
            t - \frac{-B+\sqrt{D}}{2A}
        }{
            t - \frac{-B-\sqrt{D}}{2A}
        }
        }
    \Bigg]_{t=X}^{t=Y} 
         ~+~ O(\Delta^{1/2+\varepsilon}),
\end{align}
verifying \eqref{eqn:goal-equid-A<0,D>0}.

%%%%%%%%%%%%%%%%%%%%%%%%%%%%%%%

\noindent
\textbf{Part 2: $[X,Y] \subseteq (\frac{-B+\sqrt{D}}{2A}, \frac{-B}{2A}]$. }

For this part, one can use an identical symmetry argument as in Theorem \ref{thm:equid-W--A>0,D>0} to reduce to Part 1 (for parameters $A,-B,C$).
\end{proof}

\section{Discussion}
\label{sec:discussions}

We discuss some applications of our results. First, we observe that our work yields an asymptotic formula for the number of CM points on closed geodesics in $\SL_2(\ZZ)\backslash \HH$.
\begin{proposition}
    Fix $\varepsilon > 0$.
    Then for any closed geodesic $\overline G$ in $\SL_2(\ZZ)\backslash \HH$ of discriminant $D$, the number of CM points on $\overline G$ is given by
    \begin{align}
        &\# \{\textnormal{CM points on $\overline G$ with discriminant of magnitude $\le \Delta$} \} \\
        =\,&\frac{3 \gcd(D,2)\, \mathrm{length}(\overline G) }{2\pi^2 \sqrt D} \Delta + O(\Delta^{1/2+\varepsilon}).
    \end{align}
\end{proposition}

Let $\overline G$ be an arbitrary closed geodesic in $\SL_2(\ZZ)\backslash \HH$. It is well known that $\overline G$ can be written as the image of a geodesic $G = G_{(A,B,C)}$ in $\HH$, where $(A,B,C)$ is an indefinite primitive binary quadratic form such that $D = B^2-4AC > 0$ is a non-square. In particular, this means that $G$ is a semicircle rational geodesic in $\HH$. Now, in the quotient space $\SL_2(\ZZ)\backslash \HH$, $G$ wraps around several times on itself (i.e. the quotient map from $G$ to $\overline G$ is not injective). However, it turns out that one can exactly represent $\overline G$ in $\HH$ as a subinterval along $G$ (see, for example, \cite[Theorem 13.1.6]{ratcliffe}). 
In particular, let  
$$
\gamma := 
\left(
\begin{matrix}
  \frac{t_0-Bu_0}{2} & -Cu_0 \\
  Au_0 & \frac{t_0+Bu_0}{2}
\end{matrix}
\right),
$$
where $t_0,u_0$ are the smallest positive integer solutions to the Pell equation $t^2 - D u^2 = 4$. It turns out that this $\gamma$ is a generator for $\Gamma_{(A,B,C)}$: the stabilizer of $(A,B,C)$ in $\mathrm{PSL}_2(\ZZ)$. This then means that for any $z_0 \in G$, the closed geodesic $\overline G = \Gamma_{(A,B,C)} \backslash G$ in $\SL_2(\ZZ)\backslash \HH$ can be represented exactly in $\HH$ by the arc along $G$ connecting $z_0$ and $\gamma z_0$.

One can compute the length of $\overline G$ explicity as 
\begin{align} \label{eqn:length-G-bar-sycyle-integral}
    \mathrm{length}(\overline G) = \int_{\overline G} ds_\text{hyp} = 2 \log \varepsilon_D,
\end{align}
where $\varepsilon_D = \frac 12 (t_0 + u_0 \sqrt D)$ \cite{sarnak}.
The integral over $\overline G = \Gamma_{(A,B,C)} \backslash G$ here is called a \textit{cycle integral}, and is understood to run over the arc along $G$ connecting $z_0$ and $\gamma z_0$ for some $z_0 \in G$. See, for example, \cite[Section 2]{bengoechea-imamoglu} for precise details about cycle integrals. 

On the other hand, the results of this paper imply that the cycle integral of \eqref{eqn:length-G-bar-sycyle-integral} (over a geodesic arc of $G = G_{(A,B,C)}$) can be replaced with counting the number of CM points (over the same interval). Specifically, by the argument of Theorem \ref{thm:equid-CM-geodesic}, we have that
\begin{align}
    &\# \{
    \text{CM points on $\overline G$ with discriminant of magnitude $\le \Delta$} \} \\
    &= \frac{3\Delta}{\pi^2} \int_{\overline G} d\mu(t) + O(\Delta^{1/2+\varepsilon})  \qquad \qquad\ \ \text{(evident range for $t$)}
    \\
    &= \frac{3\Delta}{\pi^2} \int_{\overline G} \frac{2}{\sqrt{16\gcd(D,2)^{-2} D}} \frac{1}{\sin \theta} d\theta + O(\Delta^{1/2+\varepsilon})
    \qquad \text{(evident range for $\theta$)}
    \\
    &= \frac{3\gcd(D,2)\Delta}{2\pi^2 \sqrt D} \int_{\overline G} \frac{1}{\sin \theta} d\theta + O(\Delta^{1/2+\varepsilon}) \\
    &= \frac{3\gcd(D,2)\,\mathrm{length}(\overline G) }{2\pi^2 \sqrt D} \Delta + O(\Delta^{1/2+\varepsilon}),
\end{align}
as claimed.

The same method can also be used to compute more general cycle integrals. Let $f$ be a modular function for $\SL_2(\ZZ)$. For every real quadratic irrationality $w \in \QQ(\sqrt{D})$, one can associate the semicircle rational geodesic $G = G_{(A,B,C)}$ connecting $w$ and its conjugate $w'$. One defines the ``value" of $f$ at $w$ to be the cycle integral
\begin{align}
    f(w) := \int_{\overline G} f(z) ds_\mathrm{hyp},
    %= \int_{\Gamma_{(a,b,c)}\backslash G} f(z) \frac{\sqrt{D} \,dz}{az^2+bz+c}.
\end{align}
where as above $\overline{G}$ denotes the image of $G$ in $\SL_2(\ZZ)\backslash \HH$.
These ``values" have been studied in various contexts. For example, \cite{duke-imamoglu-toth} relates the ``values" of the $j$-function to the Fourier coefficients of a weight $1/2$ weakly harmonic modular form, and \cite{bengoechea-imamoglu} establishes a certain convergence property of these ``values" conjectured by \cite{kaneko}.

We would like to point out here that our work yields an alternative formula for these ``values" of $f$. Specifically, by a similar argument as before, we have the following.
\begin{proposition} \label{prop:"value"-of-f-formula}
    Let $w$ be a real quadratic irrationality with associated geodesic $G=G_{(A,B,C)}$ of discriminant $D$. Denote $\CM^{* \overline G}_\Delta$ to be the set of all CM points on $\overline G = \Gamma_{(A,B,C)}\backslash G_{(A,B,C)}$ with discriminant of magnitude $\le \Delta$.
    Then the ``value" of $f$ at $w$ is given by 
    \begin{align}
        f(w) = \lim_{\Delta \to \infty} \frac{2\pi^2 \sqrt D}{3 \gcd(D,2)\Delta} 
        \sum_{
            p \in \CM^{* \overline G}_\Delta
        } f(p).
    \end{align}
\end{proposition}
This proposition means that the ``value" of $f$ at a real quadratic irrationality $w$ can be interpreted as the average value of $f$ over the CM points on the geodesic associated to $w$. It is standard to refer to $f(w)$ as the ``value" of $f$ at $w$. However, we remark here that $f(w)$ could perhaps more naturally be understood as the ``value" of $f$ at the RM curve $\langle A,B,C \rangle = G_{(A,B,C)}$.

Finally, we comment on an interesting open problem.
The results of this paper show that CM points are dense (and moreover equidistributed) along all rational geodesics in $\HH$. CM points are not dense along irrational geodesics, on the other hand, since each such geodesic can contain at most one CM point. We also give an example of a non-geodesic curve containing a dense (and moreover equidistributed) subset of CM points.
\begin{example}
    Let $L$ denote the horizontal line in $\HH$ of imaginary part $1$, and $\mathrm{CM}^L_\Delta$ denote the set of all CM points on $L$ with discriminant $|D| \le \Delta$. Then it is straighforward to verify that
    \begin{align}
        \mathrm{CM}^L_\Delta = \{ [n^2,-2mn,n^2+m^2]=\tfrac{m}{n}+i ~:~ 1 \le n \le (\Delta/4)^{1/4}, ~ m \in \mathbb{Z}, ~ \gcd(m,n)=1\}.
    \end{align}
    In particular, this means that $\CM^L_\Delta$ is equidistributed over $L$ with respect to the uniform measure $dx$ (or equivalently with respect to the hyperbolic metric $ds_\mathrm{hyp}$) as $\Delta \to \infty$.
\end{example}

It would be a very interesting problem to classify all algebraic curves in $\HH$ containing a dense (or equidistributed) subset of CM points. This would be an analog of the famous Andr\'e-Oort conjecture \cite{andre-oort}, which gives a complete classification of subvarieties of Shimura varieties containing a dense subset of special points.

\bibliographystyle{plain}
\bibliography{bibliography.bib}

\end{document}